\documentclass[12pt]{amsart}
\usepackage[top=35mm, bottom=35mm, left=25mm, right=25mm]{geometry}
\usepackage{xcolor}
\usepackage{bookmark}
\usepackage[all]{xy}  % typesetting graphs and diagrams
\usepackage{graphicx}
\usepackage{amssymb,amsmath,amsthm, amsfonts}
\usepackage{mathtools}

\usepackage{fancyhdr}
\newtheorem{thm}{Theorem}[section]
\newtheorem{lem}[thm]{Lemma}

\newtheorem{cor}[thm]{Corollary}

\theoremstyle{definition}

\newtheorem{rem}[thm]{Remark}
\newtheorem*{rem*}{Remark}

\numberwithin{equation}{section}

\begin{document}
\title{Gradient Estimates and Liouville Theorems for Lichnerowicz-type equation on Riemannian Manifolds}

\author{Youde Wang}
\address{1. School of Mathematics and Information Sciences, Guangzhou University; 2. Hua Loo-Keng Key Laboratory
of Mathematics, Institute of Mathematics, Academy of Mathematics and Systems Science, Chinese Academy
of Sciences, Beijing 100190, China; 3. School of Mathematical Sciences, University of Chinese Academy of Sciences,
Beijing 100049, China.}
\email{wyd@math.ac.cn}

\author{Aiqi Zhang}
\address{School of Mathematics and Information Sciences, Guangzhou University}
\email{zhangaiqi@gzdx.wecom.work}

%\keywords{topological entropy, entropy pair, iterated function system,  measure-theoretic entropy}
\date{\today}

\begin{abstract}
In this paper we consider the gradient estimates on positive solutions to the following elliptic (Lichnerowicz) equation defined on a complete Riemannian manifold $(M,\,g)$:
$$\Delta v + \mu v + a v^{p+1} +b v^{-q+1} =0,$$
where $p\geq-1$, $q\geq1$, $\mu$, $a$ and $b$ are real constants.

In the case $\mu\geq0$ and $b\geq0$ or $\mu<0$, $a>0$ and $b>0$ ($\mu$ has a lower bound), we employ the Nash-Moser iteration technique to obtain some refined gradient estimates of the solutions to the above equation, if $(M,\,g)$ satisfies $Ric \geq -(n-1)\kappa$, where $n\geq3$ is the dimension of $M$ and $\kappa$ is a nonnegative constant, and $\mu$, $a$, $b$, $p$ and $q$ satisfy some technique conditions. By the obtained gradient estimates we also derive some Liouville type theorems for the above equation under some suitable geometric and analysis conditions. As applications, we can derive some Cheng-Yau's type gradient estimates for solutions to the $n$-dimensional Einstein-scalar field Lichnerowicz equation where $n\geq3$.
\end{abstract}

\maketitle

\tableofcontents
%\tableofcontents
%\newpage

\section{Introduction}
In this paper we are concerned with the following elliptic equation defined on a complete Riemannian manifold $(M,\,g)$:
\begin{equation}\label{eqno1.1}
\Delta v+\mu v+ av^{p+1}+bv^{-q+1}=0,
\end{equation}
where $p\geq-1$, $q\geq1$, $\mu$, $a$ and $b$ are some real constants, $\Delta$ is the Laplace-Beltrami operator on $(M,\,g)$ with respect to the metric $g$. It is well-known that this equation is also called Lichnerowicz-type equation.

In the past half century, the above equation (\ref{eqno1.1}) with $\mu=b=0$ includes many important and well-known equations stemming from differential geometry and physics etc. These equations are deeply and extensively studied by many mathematicians. For instance, the works of Schoen and Yau in (\cite{S1, S2, SY}) on conformally flat manifold and Yamabe problem highlighted the importance of studying the distribution solutions of
\begin{equation}\label{eq:1.1*}
\Delta u(x)+u^{(n+2)/(n-2)}(x)=0,
\end{equation}
where $n\geq3$.

From the viewpoint of analysis, Caffarelli, Gidas and Spruck in \cite{CGS} studied non-negative smooth solutions of the conformal invariant equation ($\ref{eq:1.1*}$), and discussed some special form of ($\ref{eqno1.1}$), written by $\Delta u + g(u)=0$, with an isolated singularity at the origin. It is well-known that (\ref{eqno1.1}) with $\mu=b=0$ and $a=a(x)\in C^{\infty}(M)$ is called Lane-Emden equation
\begin{equation}\label{eq:1.2}
\Delta u + a(x) u^{r}=0
\end{equation}
defined on a complete noncompact Riemannian manifold $(M, g)$ where $r$ is a real constant. One also studied the connections between Liouville-type
theorems and local properties of nonnegative solutions to superlinear elliptic problems (see \cite{PQS, Sou}).

If $(M,g)$ is the standard two sphere $S^2$, this equation also has a deep relationship with the stationary solutions to Euler's equation on $S^2$ (such as zonal flows, Rossby-Haurwitz planetary waves and Stuart-type vortices). We refer to \cite{CCKW, CG, R} and references therein for recent developments of stationary solutions of Euler's equation on $S^2$.

Recently, Peng, Wang and Wei \cite{PWW} in 2021 employed suitable auxiliary functions and the maximum principles to consider the gradient estimates of the positive solution to (\ref{eq:1.2}) in the case $a>0$ and $ r <1+\frac{4}{n}$ or $a < 0$ and $r>1$ are two constants. For the case $a>0$, their results improve considerably the previous known results and supplements the results for the case $\dim(M)\leq 2$. For the case $a<0$ and $r>1$, they also improved considerably the previous related results. When the Ricci curvature of $(M,g)$ is nonnegative, a Liouville-type theorem for the above equation was established. For more details we refer to \cite{MHL,PWW} and references therein.

Very recently, Y.-D. Wang and G.-D. Wei \cite{WW} adopted the Nash-Moser iteration to study the nonexistence of positive solutions to the above Lane-Emden equation with a positive constant $a$, i.e.,
$$\Delta u+au^{r}=0$$
defined on a noncompact complete Riemannian manifold $(M, g)$ with $\dim(M)=n \geq 3$, and improve some results in \cite{PWW}. Later, inspired by the work of X.-D. Wang and L. Zhang \cite{WZ}, J. He, Y.-D. Wang and G.-D. Wei \cite{He-Wang-Wei} also discussed the gradient estimates and Liouville type theorems for the positive solution to the following generalized Lane-Emden equation
$$\Delta_pu + au^r=0.$$
Especially, the results obtained in \cite{WW} are also improved. It is shown in \cite{He-Wang-Wei} that, if the Ricci curvature of manifold is nonnegative and
$$r\in\left(-\infty,\, (n+3)/(n-1)\right)$$
the above equation with $a>0$ and $p=2$ does not admit any positive solution.

On the other hand, an analogue but more general form of Yamabe's equation is the well-known Einstein-scalar field Lichnerowicz equation. This equation arises from the Hamiltonian constraint equation for the Einstein-scalar field system in general relativity \cite{CGi, CM, Lic, York}. It is well known that the initial data for solving the nonlinear wave system of the Einstein field equations in general relativity has to satisfy the Einstein-scalar field Lichnerowicz equation. In the case the underlying manifold $M$ with dimension $n \geq 3$, the Einstein-scalar field Lichnerowicz equation takes the following form
$$\Delta u + \mu(x)u + A(x)u^r + B(x)u^{-s} = 0,$$
where $\mu(x)$, $A(x)$ and $B(x)$ are smooth functions on $M$, which are of obvious geometric or physical significance, and $r=(n+2)/(n-2)$ and $s=(3n-2)/(n-2)$; while on 2-manifolds, we know that the Einstein-scalar field Lichnerowicz equation is given as follows
$$\Delta u + A(x)e^{2u} + B(x)e^{-2u} + D(x) = 0.$$
Unless otherwise stated, solutions are always required to be smooth and positive. Applying the method used in \cite{HPP, I, IMP, MW1}, one can easily obtain the solvability of the equation provided $(M,g)$ is a compact manifold. For more details we refer to \cite{C, CGi, CM, DG, HPP, I, IMP, MW, N} and references therein. Gnerally, we call the above equation as the Einstein-scalar field Lichnerowicz-type equation if $\mu(x)$, $A(x)$ and $B(x)$ are some given functions without any geometric or physical constraints.

In particular, Song and Zhao \cite{SZ} studied the gradient estimates of positive solutions to the above Einstein-scalar field Lichnerowicz-type equation defined on a compact Riemannian manifold $(M, g)$ with $\dim(M)\geq 3$. The estimate reads as
$$|\nabla_g u|^2 \leq C(R, A(x), B(x), h(x))(u^2 + u^{r+1} + u^{1-s}),$$
which is not Cheng-Yau type estimate. The proof of the above result is in the same spirit as in \cite{MW}, where Ma and Wei studied some properties and gradient estimates of positive solutions to the equation $$\Delta u = u^\tau$$ in $\Omega$, where $\Omega\subset\mathbb{R}^N$ is a regular domain and $\tau< 0$. Moreover, in \cite{SZ} the corresponding parabolic counterpart was also considered.

Later, L. Zhao \cite{Z} discussed the Cheng-Yau type gradient estimates of the above Lichnerowicz-type equation defined on a complete Riemannian manifold $(M, g)$ with $\dim(M)\geq 3$, for more details we refer to \cite{ZY,Z}.

The parabolic counterpart of the above equation (\ref{eqno1.1}) was considered by Dung, Khan and Ng\^o \cite{DKN}. More concretely, let $(M, g, e^{-f}dv)$ be a complete, smooth metric measure space with the Bakry-\'Emery Ricci curvature bounded from below, Dung et al have ever studied the following general $f$-heat equations
$$u_t = \Delta_f u + \lambda u\ln u + \mu u + Au^r + Bu^{-s}.$$
Suppose that $\lambda$, $\mu$, $A$, $B$, $r$ and $s$ are constants with $A \leq 0$, $B\geq 0$, $r\geq 1$, and $s \geq 0$. If $u \in (0, 1]$ is a smooth solution to the above general $f$-heat equation, they obtained various gradient estimates for the bounded positive solutions, which depend on the bounds of positive solution and the Laplacian of the distance functions on domain manifolds.
Moreover, they also considered the gradient estimate of bounded positive solution $u\in [1, C)$ to the following equation on a Riemann surface
$$u_t = \Delta_f u + Ae^{2u} + Be^{-2u} + D,$$
where $A$, $B$ and $D$ are constants. Besides, Some mathematicians (see \cite{W, Y-Z}) also paid attention to a similar nonlinear parabolic equation defined on some kind of smooth metric measure space.

Recently, Huang and Wang in \cite{P. Huang & Y. Wang2022} studied the positive solutions to a class of general semilinear elliptic equations $$\Delta u(x)+uh(\ln u)=0$$ defined on a complete Riemannian manifold $(M, g)$ with $Ric(g)\geq -Kg$, and obtained the Cheng-Yau type gradient estimates of positive solutions to these equations which do not depend on the bounds of the solutions and the Laplacian of the distance function on $(M, g)$. They also obtained some Liouville-type theorems for these equations when $(M, g)$ is noncompact and $Ric(g)\geq 0$ and established some Harnack inequalities as consequences. Then, as applications of main theorem they extended their techniques to the Lichnerowicz-type equations
$$\Delta u + \mu u+ \lambda u\ln u + a u^{p+1}+ b u^{q+1}=0,$$
the Einstein-scalar field Lichnerowicz-type equations
$$\Delta u + \mu u + au^{p+1}+ bu^{q+1}=0$$
with $\dim(M)\geq 3$ and the two dimensional Einstein-scalar field Lichnerowicz equation
\begin{equation}\label{eq:1.4}
\Delta u + A e^{2u} + B e^{-2u} + D = 0,
\end{equation}
and obtained some similar gradient estimates and Liouville theorem under some suitable analysis conditions on these equations.

It seems that except for the results in \cite{P. Huang & Y. Wang2022, Z} there is no any other Cheng-Yau type gradient estimates and Liouville type theorem on the solution to the Einstein-scalar field Lichnerowicz type equation defined on a noncompact complete Riemannian manifold with $\dim(M)\geq 3$ by the authors' best knowledge. Actually, the gradient estimates established in \cite{P. Huang & Y. Wang2022, Z} are not exactly of the Cheng-Yau type. So, in this paper we are intended to establish some more explicit and concise gradient estimates of obvious geometric significance and then derive consequently some Liouville type theorems.

In this paper we only focus on \eqref{eqno1.1} to discuss the gradient estimate of its positive solutions and the nonexistence of its solutions by Nash-Moser iteration method and some analytical technique adopted in \cite{He-Wang-Wei}, although it is of independent interest that one studies various properties of solutions to a class of equations $\Delta u + u\tilde{h}(x, \ln u)=0$ defined on a complete Riemannian manifold and our method employed here is also used to study this class of equations.

Moreover, we can also use the Nash-Moser iteration method to deal with the following
\begin{equation*}
\Delta_m v+\mu v+ av^{p+1}+bv^{-q+1}=0,
\end{equation*}
where $\Delta_m$ ($m>1$) is the $m$-Laplace operator, and deduce some similar results with that for \eqref{eqno1.1}. We will discuss the gradient estimates and some related properties for its solutions in a forthcoming paper.
\medskip

In the sequel, we always let $(M,\,g)$ be a complete Riemannian manifold with Ricci curvature $Ric \geq -(n-1)\kappa$. For the sake of convenience, we need to make some conventions firstly. Throughout this paper, unless otherwise mentioned, we always assume $\kappa\geq0$, $n\geq3$ is the dimension of $M$, $\mu$, $a$, $b$, $p\geq-1$, and $q\geq1$ are some real constants. Moreover, we denote:
$$B_{r}=B(o,\,r),~\mbox{any}~r>0.$$

Now, we are ready to state our results.

\begin{thm}
\label{thm1}
Let $(M,\,g)$ be a complete Riemannian manifold with $Ric \geq -(n-1)\kappa$. Assume that $v$ is a smooth positive solution of (\ref{eqno1.1}) on the geodesic ball $B_R\subset M$ with $q\geq 1$, $\mu\geq0$ and $b\geq0$.
If one of the  following conditions holds true:
\begin{enumerate}
\item  $a\left(\frac{2}{n-1}-p\right)\geq0\quad\text{and}\quad p\geq -1$;
\item  $0< p< \frac{4}{n-1}$,
\end{enumerate}
then we  have:
\begin{equation}\label{eqno1.2}
\frac{\left|\nabla v\right|^2}{v^2}\leq c(n,\,p)\frac{\left(1+\sqrt\kappa R\right)^2}{R^2}\quad\mbox{on}~B_{R/2}.
\end{equation}
\end{thm}
Immediately, we have the following direct corollary:
\begin{cor}\label{cor1}
Let $(M,\,g)$ be a noncompact complete Riemannian manifold with nonnegative Ricci curvature. The equation (\ref{eqno1.1}) with $q\geq 1$, $\mu\geq0$ and $b\geq0$ does not admit any positive solution if one of the following conditions holds true:
\begin{enumerate}
\item $a>0$ and $-1\leq p <\frac{4}{n-1}$;
\item $a=0$, $\mu+b\neq0$, $p\geq -1$;
\item $a<0$, $b=\mu=0$ and $p>0$.
\end{enumerate}

Moreover, we can conclude that any positive solutions to (\ref{eqno1.1}) must be a constant if $b+\mu\neq0$ and $p>0$ in the case \eqref{eqno1.1} with $\mu\geq0$, $a<0$, $b\geq0$ and $q\geq1$.
\end{cor}

\begin{rem}
It is worthy to point out that the above results have been established in \cite{He-Wang-Wei} if $\mu=b=0$ in the equation \eqref{eqno1.1}.
\end{rem}

On the other hands, we consider (1.1) with $\mu<0$, and we also get the following results:
\begin{thm}\label{thm2}
Let $(M,\,g)$ be a complete Riemannian manifold with $Ric \geq -(n-1)\kappa$.
Assume that $v$ is a smooth positive solution of (\ref{eqno1.1})on the geodesic ball $B_R\subset M$. For any $q\geq1$, $a>0$, and $b>0$, if
$$\max\{-a,\,-b\}<\mu<0\quad\mbox{and}\quad 0< p<\frac{4}{n-1}\left(1-\frac{\mu}{\max\{-a,\,-b\}}\right),$$
then there holds true on $B_{R/2}$
\begin{equation}\label{eqno1.5}
\frac{\left|\nabla v\right|^2}{v^2}\leq  c\left(n,\,p,\,\frac{\mu}{\max\{-a,\,-b\}}\right)\frac{\left(1+\sqrt\kappa R\right)^2}{R^2},\
\end{equation}
where $c\left(n,\,p,\,\frac{\mu}{\max\{-a,\,-b\}}\right)$ is a positive constant which depends on $n$, $p$ and $\frac{\mu}{\max\{-a,\,-b\}}$.
\end{thm}

\begin{cor}\label{cor2}
Let $(M,\,g)$ be a noncompact complete Riemannian manifold with nonnegative Ricci curvature. The equation (\ref{eqno1.1}) with $a>0,\,b>0$, $\max\{-a,\,-b\}<\mu<0$,  $q\geq 1$, and $0\leq p<\frac{4}{n-1}\left(1-\frac{\mu}{\max\{-a,\,-b\}}\right)$ does not admit any positive solution.
\end{cor}

Now, if we let $p=\frac{4}{n-2}$ and $q=\frac{4(n-1)}{n-2}$ in (\ref{eqno1.1}), then (\ref{eqno1.1}) is exactly the Einstein-scalar field Lichnerowicz equation. It is easy to see that for this special case
$$p=\frac{4}{n-2}>\frac{4}{n-1}\quad\mbox{and}\quad q=\frac{4(n-1)}{n-2}\in(1,\,+\infty).$$
Hence, as a direct corollary of the above Theorem \ref{thm1}, Corollary \ref{cor1}, Theorem \ref{thm2} and Corollary \ref{cor2}, we have

\begin{thm}\label{thm3}
{\bf(Einstein-scalar field Lichnerowicz-type equation)}
Let $(M,\,g)$ be a $n$-dimensional complete Riemannian manifold with $Ric\geq-(n-1)\kappa$. Assume that $v$ is a smooth positive solution of the Einstein-scalar field Lichnerowicz-type equation on the geodesic ball $B_R\subset M$. For any $\mu\geq0$, $a\leq0$ and $b\geq0$, on $B_{R/2}$ there holds
$$
\frac{\left|\nabla v\right|^2}{v^2}\leq  c(n)\frac{\left(1+\sqrt\kappa R\right)^2}{R^2},
$$
where $c(n)$ is the same as in Theorem \ref{thm1} with $p=\frac{4}{n-2}$. Moreover, if $(M,\,g)$ is a noncompact complete manifold with nonnegative Ricci curvature, then any positive solution to the Einstein-scalar field Lichnerowicz equation with $\mu\geq0$, $a\leq0$ and $b\geq0$ is a constant.
\end{thm}

In particular, in the case the Einstein-scalar field Lichnerowicz-type equation with $a=0$ or $b=\mu=0$ we can deduce the following conclusion:
\begin{thm}\label{thm4}
{\bf(Einstein-scalar field Lichnerowicz-type equation)}
Let $(M,\,g)$ be a n-dimensional noncompact complete Riemannian manifold with a nonnegative Ricci curvature. The Einstein-scalar field Lichnerowicz-type equation does not admit any positive solution if one of the following conditions is satisfied:
\begin{enumerate}
\item $a=0$, $\mu\geq0$, $b\geq0$, and $\mu+b\neq0$;
\item $a<0$ and $\mu=b=0$.
\end{enumerate}
\end{thm}

\begin{rem}In comparison with the related conclusions stated in \cite{P. Huang & Y. Wang2022}, our local gradient estimates for solutions to the $n$-dimensional Einstein-scalar field Lichnerowicz equation ($n\geq3$) with $\mu\geq0$, $a\leq0$ and $b\geq0$
are of a similar form with Cheng-Yau's gradient estimate for positive harmonic functions on Riemannian manifolds (cf.\cite{SY}). An important feature of this type estimate is that it depends only on $n$, $\kappa$ and $R$. Therefore, one does not need to take account of any other parameters which need to be calculated separately.
\end{rem}

As a direct consequence, for some special situations we also obtain some local gradient estimates on the positive solutions to the Einstein-scalar field Lichnerowicz equation arose in relativity and some Liouville properties of this equation under some geometric conditions (see Section 5).

The rest of this paper is  organized as follows. In Section 2, we will give a detailed estimate of Laplacian of $$|\nabla\ln v|^2=\dfrac{|\nabla v|^2}{v^2},$$ where $v$ is a positive solution of (\ref{eqno1.1}) with $\mu\geq0$ and $b\geq0$. Then we need to recall the Saloff-Coste's Sobolev embedding theorem. In Section 3, we use the Moser iteration to prove Theorem \ref{thm1}, then we give the proof of Corollary \ref{cor1}. In Section 4, we discuss (\ref{eqno1.1}) with $\mu<0$ and $b\geq0$ under some suitable analysis constraints by almost the same method as in Section 3,  and infer a similar conclusion for the $n$-dimensional Einstein-scalar field Lichnerowicz-type equation. In Section 5, we recall some backgrounds on Lichnerowicz equaiton in relativity, then we give some applications and comments of Theorem \ref{thm3}.

\section{Prelimanary}
Now, let $v$ be a positive smooth solution to the elliptic equation:
\begin{equation}\label{eqno2.1}
\Delta v + \mu v + a v^{p+1} +b v^{-q+1} =0\quad\mbox{on}~ B_R,
\end{equation}
where $\mu$, $a$, $b$, $p\geq-1$ and $q\geq1$ are some real constants. Set $u=-\ln v$. We compute directly to obtain
\begin{equation}\label{eqno2.2}
\Delta u=|\nabla u|^2+\mu+ae^{-pu}+be^{qu}.
\end{equation}
For convenience, we denote $f=|\nabla u|^2$. By a direct calculation we can verify
\begin{equation}
\label{eqno2.3}
\Delta u=f+\mu+ae^{-pu}+be^{qu}.
\end{equation}

Firstly, we need to establish the following lemmas.

\begin{lem}\label{lem2.1}
Let $v$ be a positive smooth solution to \eqref{eqno2.1}, $u=-\ln v$ and $f=|\nabla u|^2$. For any real numbers $\mu,\,a,\,b$, $q\geq1$ and $p\geq-1$ which are associated with \eqref{eqno2.1}, then, at the point where $f\neq0$,  for any given $\iota\geq1$ we have:
\begin{equation}\label{eqno2.4}
\begin{aligned}
\frac{\Delta \left(f^\iota\right)}{\iota f^{\iota-1}}
\geq&-2(n-1)\kappa f+\frac{2f^2}{n-1}+\frac{2(n-2)\langle\nabla u,\,\nabla f\rangle}{n-1}\\
&+\frac{2(2\iota-1)}{2(\iota-1)(n-1)+n}\left(\mu+ae^{-pu}+be^{qu}\right)^2\\
&+4f\frac{\mu+ae^{-pu}+be^{qu}}{n-1}+2f \left( bqe^{qu}-ape^{-pu} \right).
\end{aligned}
\end{equation}
\end{lem}

\begin{proof}
For any $q\geq1$ and $p\geq-1$, by Bochner formula we have
\begin{equation}\label{eqno2.5}
\begin{aligned}
\Delta f
= & 2|\nabla ^2 u|^2 + 2 Ric(\nabla u,\,\nabla u)+ 2\langle \nabla u,\,\nabla \Delta u\rangle\\
\geq& 2|\nabla ^2 u|^2 -2(n-1)\kappa f+ 2\langle \nabla u,\,\nabla f\rangle + 2f \left( bqe^{qu}-ape^{-pu} \right).
\end{aligned}
\end{equation}

Now, we choose a suitable local orthonormal frame $\left\{\xi_i \right\}_{i=1}^{n}$ such that $\nabla u=|\nabla u|\xi_1$. If
we denote $\nabla u=\sum_{i=1}^{n}u_i\xi_i$, it is easy to see
$$
u_1=|\nabla u|\quad\mbox{and}\quad u_i=0
$$
for any $2 \leq i \leq n$. Noticing that
$$
\sum_{i=1}^{n} u_{ii}=f+\mu+ae^{-pu}+be^{qu},
$$
we have
$$
\begin{aligned}
|\nabla ^2 u|^2
\geq& u_{11}^{2}+\sum_{i=2}^{n} u_{ii}^2\\
\geq& u_{11}^{2}+\frac{1}{n-1} \left( f+\mu+ae^{-pu}+be^{qu}-u_{11}\right)^2\\
=&\frac{f^2}{n-1}-\frac{2fu_{11}}{n-1}+\frac{n}{n-1}u_{11}^{2}+\frac{\left(\mu+ae^{-pu}+be^{qu}\right)^2}{n-1}\\
&-2u_{11}\frac{\mu+ae^{-pu}+be^{qu}}{n-1}+2f\frac{\mu+ae^{-pu}+be^{qu}}{n-1}.
\end{aligned}
$$
Since
$$
\begin{aligned}
2f u_{11}
=&2|\nabla u|^2\nabla^2u(\xi_1,\,\xi_1)\\
=&2|\nabla u|^2\langle\nabla_{\xi_1}du, \xi_1\rangle\\
=&2|\nabla u|^2\left[\xi_1\left(|\nabla u|\right)-\left(\nabla_{\xi_1}\xi_1\right)u\right]\\
=&2|\nabla u|^2\langle \xi_1,\,\nabla|\nabla u|\rangle\\
=&\langle\nabla u,\,\nabla f\rangle,
\end{aligned}
$$
we have
$$
\begin{aligned}
|\nabla ^2 u|^2
\geq&\frac{f^2}{n-1}-\frac{\langle\nabla u,\,\nabla f\rangle}{n-1}+\frac{n}{n-1}u_{11}^{2}+\frac{\left(\mu+ae^{-pu}+be^{qu}\right)^2}{n-1}\\
& -2u_{11}\frac{\mu+ae^{-pu}+be^{qu}}{n-1}+2f\frac{\mu+ae^{-pu}+be^{qu}}{n-1}.
\end{aligned}
$$
Hence, by substituting the above inequality into (\ref{eqno2.5}) we get
$$
\begin{aligned}
\Delta f
\geq & \frac{2f^2}{n-1}-\frac{2\langle\nabla u,\,\nabla f\rangle}{n-1}+\frac{2n}{n-1}u_{11}^{2}+\frac{2\left(\mu+ae^{-pu}+be^{qu}\right)^2}{n-1} -4u_{11}\frac{\mu+ae^{-pu}+be^{qu}}{n-1}\\
&+4f\frac{\mu+ae^{-pu}+be^{qu}}{n-1}-2(n-1)\kappa f+ 2\langle \nabla u,\,\nabla f\rangle + 2f \left( bqe^{qu}-ape^{-pu} \right)\\
=&-2(n-1)\kappa f+\frac{2f^2}{n-1}+\frac{2(n-2)\langle\nabla u,\,\nabla f\rangle}{n-1}+\frac{2n}{n-1}u_{11}^{2}+\frac{2\left(\mu+ae^{-pu}+be^{qu}\right)^2}{n-1} \\
&-4u_{11}\frac{\mu+ae^{-pu}+be^{qu}}{n-1}+4f\frac{\mu+ae^{-pu}+be^{qu}}{n-1} + 2f \left( bqe^{qu}-ape^{-pu} \right).
\end{aligned}
$$
For any $\iota\geq1$, we have
$$
\Delta \left(f^\iota\right)=\iota\left(\iota-1\right)f^{\iota-2}|\nabla f|^2+\iota f^{\iota-1}\Delta f,
$$
therefore,
\begin{equation}\label{eqno2.6}
\frac{\Delta \left(f^\iota\right)}{\iota f^{\iota-1}}=(\iota-1)f^{-1}|\nabla f|^2+\Delta f.
\end{equation}
Since
$$
|\nabla f|^2=\sum_{i=1}^{n}|2u_1u_{1i}|^2=4f\sum_{i=1}^{n}u_{1i}^2\geq4fu_{11}^2,
$$
it is not difficult to see that
\begin{equation}\label{eqno2.7}
\begin{aligned}
\frac{\Delta \left(f^\iota\right)}{\iota f^{\iota-1}}
\geq&-2(n-1)\kappa f+\frac{2f^2}{n-1}+\frac{2(n-2)\langle\nabla u,\,\nabla f\rangle}{n-1}+\left[4(\iota-1)+\frac{2n}{n-1}\right]u_{11}^{2}\\
&+\frac{2\left(\mu+ae^{-pu}+be^{qu}\right)^2}{n-1}-4u_{11}\frac{\mu+ae^{-pu}+be^{qu}}{n-1}\\
&+4f\frac{\mu+ae^{-pu}+be^{qu}}{n-1}+2f \left( bqe^{qu}-ape^{-pu} \right).
\end{aligned}
\end{equation}
Since
$$
\begin{aligned}
&\left[4(\iota-1)+\frac{2n}{n-1}\right]u_{11}^{2}-4u_{11}\frac{\mu+ae^{-pu}+be^{qu}}{n-1}\\
=&\frac{4(\iota-1)(n-1)+2n}{n-1}\left[u_{11}^{2}-4u_{11}\frac{\mu+ae^{-pu}+be^{qu}}{n-1}\frac{n-1}{4(\iota-1)(n-1)+2n}\right]\\
=&\frac{4(\iota-1)(n-1)+2n}{n-1}\left[u_{11}-\frac{\mu+ae^{-pu}+be^{qu}}{2(\iota-1)(n-1)+n}\right]^2\\
&-\frac{4(\iota-1)(n-1)+2n}{n-1}\left[\frac{\mu+ae^{-pu}+be^{qu}}{2(\iota-1)(n-1)+n}\right]^2\\
\geq&-\frac{4(\iota-1)(n-1)+2n}{n-1}\left[\frac{\mu+ae^{-pu}+be^{qu}}{2(\iota-1)(n-1)+n}\right]^2\\
=&-\frac{1}{2(\iota-1)(n-1)+n}\cdot\frac{2}{n-1}\left(\mu+ae^{-pu}+be^{qu}\right)^2,
\end{aligned}
$$
and
$$
\begin{aligned}
\frac{2}{n-1}-\frac{1}{2(\iota-1)(n-1)+n }\cdot\frac{2}{n-1}
=&\frac{2}{n-1}\frac{2(\iota-1)(n-1)+n -1}{2(\iota-1)(n-1)+n }\\
=&\frac{2(2\iota-1)}{2(\iota-1)(n-1)+n},
\end{aligned}
$$
then,
$$
\begin{aligned}
\frac{\Delta \left(f^\iota\right)}{\iota f^{\iota-1}}
\geq&-2(n-1)\kappa f+\frac{2f^2}{n-1}+\frac{2(n-2)\langle\nabla u,\,\nabla f\rangle}{n-1}\\
&+\frac{2(2\iota-1)}{2(\iota-1)(n-1)+n}\left(\mu+ae^{-pu}+be^{qu}\right)^2\\
&+4f\frac{\mu+ae^{-pu}+be^{qu}}{n-1}+2f \left( bqe^{qu}-ape^{-pu} \right),
\end{aligned}
$$
which finishes the proof of this lemma.
\end{proof}

\begin{lem}\label{lem2.2}
Let $v$ be a positive solution to \eqref{eqno2.1} and $f$ be the same as in Lemma \ref{lem2.1}. Suppose that $\mu\geq0$ and $b\geq0$ in \eqref{eqno2.1}. Then, for any $q\geq1$ and given $\iota\geq1$, at the point where $f\neq0$ we have:
\begin{enumerate}
\item  if $a\left(\frac{2}{n-1}-p\right)\geq0$ and $p\geq -1$, there holds
\begin{equation}\label{eqno2.8}
\frac{\Delta \left(f^\iota\right)}{\iota f^{\iota-1}}\geq-2(n-1)\kappa f+ \frac {2(n-2)}{n-1}\langle \nabla u,\nabla f\rangle\\+\frac{2}{n-1}f^2;
\end{equation}
\item if $p\geq0$, there holds
\begin{equation}\label{eqno2.9}
\frac{\Delta \left(f^\iota\right)}{\iota f^{\iota-1}}\geq-2(n-1)\kappa f+ \frac {2(n-2)}{n-1}\langle \nabla u,\nabla f\rangle\\+\rho(n,\,p,\,\iota)f^2.
\end{equation}
where
$\rho(n,\,p,\,\iota)$ is a constant which depends on $n$, $p$ and $\iota$.
\end{enumerate}
\end{lem}

\begin{proof}
For any $q \geq 1$, we need to deal with the following two cases respectively:
 \medskip

 (1). By assumptions $\mu\geq0$ and $b\geq0$, Lemma \ref{lem2.1} tells us that at the point where $f\neq0$ there holds
$$
\frac{\Delta \left(f^\iota\right)}{\iota f^{\iota-1}}
\geq-2(n-1)\kappa f+\frac{2f^2}{n-1}+\frac{2(n-2)\langle\nabla u,\,\nabla f\rangle}{n-1}+2ae^{-pu}f \left(\frac{2}{n-1}-p\right).
$$
Obviously, if $$a\left(\frac{2}{n-1}-p\right)\geq0,\, p\geq-1$$
 we can deduce that (\ref{eqno2.8}) holds true.

(2). On the other hand, since $q\geq1\geq-p$, it follows from (\ref{eqno2.4}) that
$$
\begin{aligned}
\frac{\Delta \left(f^\iota\right)}{\iota f^{\iota-1}}
\geq&-2(n-1)\kappa f+\frac{2f^2}{n-1}+\frac{2(n-2)\langle\nabla u,\,\nabla f\rangle}{n-1}\\
&+\frac{2(2\iota-1)}{2(\iota-1)(n-1)+n}\left(\mu+ae^{-pu}+be^{qu}\right)^2\\
&+4f\frac{\mu+ae^{-pu}+be^{qu}}{n-1}-2pf \left(ae^{-pu}+ be^{qu} \right)\\
=&-2(n-1)\kappa f+\frac{2f^2}{n-1}+\frac{2(n-2)\langle\nabla u,\,\nabla f\rangle}{n-1}\\
&+\frac{2(2\iota-1)}{2(\iota-1)(n-1)+n}\left(\mu+ae^{-pu}+be^{qu}\right)^2\\
&+4f\frac{\mu+ae^{-pu}+be^{qu}}{n-1}-2pf \left(\mu+ae^{-pu}+ be^{qu} \right)+2pf\mu.
\end{aligned}
$$
Therefore, if $p\geq0$, we obtain
$$
\begin{aligned}
\frac{\Delta \left(f^\iota\right)}{\iota f^{\iota-1}}
\geq&-2(n-1)\kappa f+\frac{2f^2}{n-1}+\frac{2(n-2)\langle\nabla u,\,\nabla f\rangle}{n-1}\\
&+\frac{2(2\iota-1)}{2(\iota-1)(n-1)+n}\left(\mu+ae^{-pu}+be^{qu}\right)^2\\
&+2f\left(\mu+ae^{-pu}+be^{qu}\right)\left(\frac{2}{n-1}-p\right).
\end{aligned}
$$
Since
$$
\begin{aligned}
&\frac{2(2\iota-1)}{2(\iota-1)(n-1)+n}\left(\mu+ae^{-pu}+be^{qu}\right)^2+2f \left(\mu+ae^{-pu}+be^{qu}\right)\left(\frac{2}{n-1}-p\right)\\
=&\frac{2(2\iota-1)}{2(\iota-1)(n-1)+n}\left[\left(\mu+ae^{-pu}+be^{qu}\right)^2+2 f\left(\mu+ae^{-pu}+be^{qu}\right)\left(\frac{2}{n-1}-p\right)\right.\\
&\left.\times\frac{2(\iota-1)(n-1)+n}{2(2\iota-1)}\right]\\
=&\frac{2(2\iota-1)}{2(\iota-1)(n-1)+n}\left[\left(\mu+ae^{-pu}+be^{qu}\right)+f\left(\frac{2}{n-1}-p\right)\frac{2(\iota-1)(n-1)+n}{2(2\iota-1)}\right]^2\\
&-\frac{2(2\iota-1)}{2(\iota-1)(n-1)+n}\left[f\left(\frac{2}{n-1}-p\right)\frac{2(\iota-1)(n-1)+n}{2(2\iota-1)}\right]^2\\
\geq&-\frac{2(\iota-1)(n-1)+n}{2(2\iota-1)}\left(\frac{2}{n-1}-p\right)^2f^2,
\end{aligned}
$$
then,
$$
\begin{aligned}
\frac{\Delta \left(f^\iota\right)}{\iota f^{\iota-1}}
\geq&-2(n-1)\kappa f+\frac{2f^2}{n-1}+\frac{2(n-2)\langle\nabla u,\,\nabla f\rangle}{n-1}\\
&-\frac{2(\iota-1)(n-1)+n}{2(2\iota-1)}\left(\frac{2}{n-1}-p\right)^2f^2.
\end{aligned}
$$
Therefore, by rearranging the above we arrive at
$$
\begin{aligned}
\frac{\Delta \left(f^\iota\right)}{\iota f^{\iota-1}}
\geq&-2(n-1)\kappa f+\frac{2(n-2)\langle\nabla u,\,\nabla f\rangle}{n-1}\\
&+\left[\frac{2}{n-1}-\frac{2(\iota-1)(n-1)+n}{2(2\iota-1)}\left(\frac{2}{n-1}-p\right)^2\right]f^2.
\end{aligned}
$$
Now we denote that
$$
\rho (n,\,p,\,\iota)=\frac{2}{n-1}-\frac{2(\iota-1)(n-1)+n}{2(2\iota-1)}\left(\frac{2}{n-1}-p\right)^2,
$$
where $p\geq0$. It is easy to see that (\ref{eqno2.9}) has been established. Thus, the proof of Lemma \ref{lem2.2} is completed.\end{proof}

\begin{lem}\label{lem2.3}
Let $v$ is a positive solution to \eqref{eqno2.1} and $f$ be the same as in Lemma \ref{lem2.1}. Suppose that $\mu\geq0$, $b\geq0$ and $q\geq1$ in equation \eqref{eqno2.1}. If one of the following conditions holds true:
\begin{enumerate}
\item  $a\left(\frac{2}{n-1}-p\right)\geq0$ and  $p\geq -1$;
\item  $0< p< \frac{4}{n-1}$,
\end{enumerate}
then, there exist some constant $\iota>1$ large enough such that at the point where $f\neq0$ there holds true
\begin{equation}\label{eqno2.10}
\frac{\Delta \left(f^\iota\right)}{\iota f^{\iota-1}}\geq-2(n-1)\kappa f+ \frac {2(n-2)}{n-1}\langle \nabla u,\nabla f\rangle\\+\tilde\rho(n,\,p)f^2,
\end{equation}
where $\tilde\rho(n,\,p)\in\left(0,\,\frac{2}{n-1}\right]$ depends only on $n$ when $a\left(\frac{2}{n-1}-p\right)\geq0$, and depends on both $n$ and $p$ when $0< p< \frac{4}{n-1}$.
\end{lem}

\begin{proof} By assumptions we have $\mu\geq0$ and $b\geq0$. For any $q\geq1$ and $\iota\geq1$, we deal with the following two cases respectively:
\medskip

{\bf Case 1}:      $$a\left(\frac{2}{n-1}-p\right)\geq0\quad\text{and}\quad p\geq -1.$$

For this case, we let $$\tilde\rho(n,\,p)=\frac{2}{n-1}.$$
Then, according to (\ref{eqno2.8}) we can infer the required result is true for any $\iota\geq1$.
\medskip

{\bf Case 2}:       $$0< p< \frac{4}{n-1}.$$

In the present situation, we have
$$ 0\leq \left|\frac{2}{n-1}-p\right|<\frac{2}{n-1}\quad\quad\mbox{and}\quad\quad 0\leq\left(\frac{2}{n-1}-p\right)^2<\frac{4}{(n-1)^2}.$$
Noting that
$$
\lim\limits_{\iota\to+\infty}\frac{2(\iota-1)(n-1)+n}{2(2\iota-1)}=\frac{n-1}{2},
$$
and the monotonicity of the following function
$$y(\iota)=\frac{2(\iota-1)(n-1)+n}{2(2\iota-1)}$$
with respect to $\iota$ on $[1,\,\infty)$, we obtain that for any $\iota\geq1$ there holds true
$$
\frac{n-1}{2}\leq\frac{2(\iota-1)(n-1)+n}{2(2\iota-1)}\leq\frac{n}{2}.
$$
Hence, we choose $\iota=\iota_{n,\,p}$ large enough such that
$$
\frac{2(\iota-1)(n-1)+n}{2(2\iota-1)}\left(\frac{2}{n-1}-p\right)^2<\frac{2}{n-1}
$$
and
$$
\rho(n,\,p,\,\iota)=\frac{2}{n-1}-\frac{2(\iota-1)(n-1)+n}{2(2\iota-1)}\left(\frac{2}{n-1}-p\right)^2>0.
$$
Now, we let $\tilde\rho(n,\,p)=\rho(n,\,p,\,\iota_{n,p})$, then the required inequality follows immediately. Thus, the proof of Lemma \ref{lem2.3} is completed.\end{proof}

Next, we need to recall the Saloff-Coste's Sobolev embedding theorem (Theorem 3.1 in \cite{L. Saloff-Coste1992}, which plays a key role on the arguments (Moser iteration) taken here.

\begin{thm}\label{thm Sobolev}
(the Saloff-Coste's Sobolev embedding theorem)
Let $(M,\,g)$ be a complete Riemannian manifold with $Ric\geq -(n-1)\kappa $.
For any $n>2$, there exist a constant $c_n$, depending only on $n$, such that for all $B\subset M$ we have
$$
\left(\int_{B}f^{2\lambda}\right)^{\frac{1}{\lambda}}\leq e^{c_n\left(1+\sqrt \kappa R\right)}V^{-\frac{2}{n}}R^2\left(\int_{B}\left|\nabla f\right|^2+R^{-2}\int_{B}f^2\right),
\quad f\in C_{0}^{\infty}(B),
$$
where $R$ and $V$ are the radius and volume of $B$, constant $\lambda=\frac{n}{n-2}$.
For $n=2$, the above inequality holds with $n$ replaced by any fixed $n'>2$.
\end{thm}

\section{Proof of Theorem \ref{thm1}}
In this section, we focus on the proof of Theorem \ref{thm1} and Corollary \ref{cor1} on the equation (\ref{eqno1.1}) in the case $\mu\geq0$ and $b\geq0$.
Throughout this section, unless otherwise mentioned, $\mu$ and $b$ are two nonnegative constants.

\begin{lem} Let $v$ be a positive solution to \eqref{eqno1.1}, $u=-\ln v$ and $f=|\nabla u|^2$ as before. Assume that one of the following two conditions holds true:
\begin{enumerate}
\item  $a\left(\frac{2}{n-1}-p\right)\geq0$ and $p\geq-1$;
\item  $0< p< \frac{4}{n-1}$.
\end{enumerate}
Then, there exists $\theta_0=c_{n,p}(1+\sqrt{\kappa}R)$, where $c_{n, p}=\max\{c_n,\,2\iota,\,\frac{16}{\tilde\rho(n,\,p)}\}$ is a positive constant depending on $n$, $p$ and $\tilde\rho(n,\,p)$ which is defined as before, such that for any $0\leq\eta\in C_{0}^{\infty}(B_{R})$ and any $\theta\geq\theta_0$ large enough there holds true
\begin{equation*}
\begin{aligned}
&e^{-\theta_0}V^{\frac{2}{n}}\left(\int_{B_{R}}f^{(\theta+1)\lambda}\eta^{2\lambda}\right)^{\frac{1}{\lambda}}+
4\theta\tilde\rho(n,\,p)R^2\int_{B_{R}}f^{\theta+2}\eta^2\\
\leq&\theta_0^2\theta \int_{B_{R}} f^{\theta+1}\eta^2 +66R^2\int_{B_{R}}f^{\theta+1}\left|\nabla\eta\right|^2.
\end{aligned}
\end{equation*}
Here $B_R$ is a geodesic Ball in $(M, g)$ and $V$ is the volume of $B_R$.
\end{lem}

\begin{proof} First, we note the fact $f$ is smooth enough away from $\{x: f(x)=0\}$ according to the standard elliptic regularity theory.
Let
$$ A=\left\{x\in B_R|f(x)=0\right\},\quad \tilde A=B_R\setminus A.$$
Therefore, according to the Lemma \ref{lem2.3}, we take integration by part to derive that,
for any $q\geq 1$ and any function $\varphi\in W_{0}^{1,2}(B_R)$ with $\varphi\geq 0$ and $\mbox{supp}(\varphi)\subset\subset\tilde A$,
there holds true:
\begin{equation}\label{add}
\int_{B_{R}}\frac{\Delta \left(f^\iota\right)}{\iota f^{\iota-1}}\varphi
\geq-2(n-1)\kappa \int_{B_{R}}f\varphi +\frac{2\left(n-2\right)}{n-1}\int_{B_{R}}\langle\nabla u,\,\nabla f\rangle \varphi+\tilde\rho(n,\,p)\int_{B_{R}}f^{2}\varphi,
\end{equation}
where $\iota\geq1$ is a suitable positive constant chosen in the proof of Lemma \ref{lem2.3}.

From (\ref{eqno2.6}) we have
$$
\int_{B_{R}}\Delta f\varphi=\int_{B_{R}}\frac{\Delta \left(f^\iota\right)}{\iota f^{\iota-1}}\varphi-(\iota-1)\int_{B_{R}}f^{-1}|\nabla f|^2\varphi.
$$
Substituting \eqref{add} into the above identity leads to
$$
\begin{aligned}
\int_{B_{R}}\Delta f\varphi
\geq&-2(n-1)\kappa \int_{B_{R}}f\varphi +\frac{2\left(n-2\right)}{n-1}\int_{B_{R}}\langle\nabla u,\,\nabla f\rangle \varphi\\
&+\tilde\rho(n,\,p)\int_{B_{R}}f^{2}\varphi-(\iota-1)\int_{B_{R}}f^{-1}|\nabla f|^2\varphi.
\end{aligned}
$$
Hence, it follows
\begin{equation}\label{eqno3.1}
\begin{aligned}
\int_{B_{R}}\langle \nabla f,\, \nabla\varphi \rangle\
\leq&2(n-1)\kappa \int_{B_{R}}f\varphi -\frac{2\left(n-2\right)}{n-1}\int_{B_{R}}\langle\nabla u,\,\nabla f\rangle \varphi\\
&-\tilde\rho(n,\,p)\int_{B_{R}}f^{2}\varphi+(\iota-1)\int_{B_{R}}f^{-1}|\nabla f|^2\varphi.
\end{aligned}
\end{equation}

Now, for any $\epsilon>0$ we define
$$f_\epsilon =\left(f-\epsilon\right)^+.$$
Let $\varphi=\eta^2 f_\epsilon^\theta\in W_{0}^{1,2}(B_R)$ where $0\leq\eta\in C_{0}^{\infty}(B_{R})$ and $\theta>\max\{1,\,2(\iota-1)\}$ will be determined later.
Direct computation shows that
$$
\nabla \varphi= 2f_\epsilon^\theta\eta\nabla\eta+\theta f_\epsilon^{\theta-1}\eta^2\nabla f.
$$
By substituting the above into (\ref{eqno3.1}), we derive
$$
\begin{aligned}
&\int_{B_{R}} \langle\nabla f,\,2f_\epsilon^\theta\eta\nabla\eta+\theta f_\epsilon^{\theta-1}\eta^2\nabla f\rangle\\
\leq& 2(n-1)\kappa\int_{B_{R}}f\eta^2 f_\epsilon^\theta -\frac{2\left(n-2\right)}{n-1}\int_{B_{R}}\langle\nabla u,\,\nabla f\rangle  \eta^2 f_\epsilon^\theta-\tilde\rho(n,\,p)\int_{B_{R}}f^2\eta^2 f_\epsilon^\theta\\
&+(\iota-1)\int_{B_{R}}f^{-1}|\nabla f|^2\eta^2 f_\epsilon^\theta,
\end{aligned}
$$
it follows that
$$
\begin{aligned}
&2\int_{B_{R}} f_\epsilon^\theta  \eta\langle\nabla f,\,\nabla\eta\rangle+\theta\int_{B_{R}} f_\epsilon^{\theta-1} |\nabla f|^2\eta^2\\
\leq& 2(n-1)\kappa\int_{B_{R}}f\eta^2 f_\epsilon^\theta -\frac{2\left(n-2\right)}{n-1}\int_{B_{R}}\langle\nabla u,\,\nabla f\rangle  \eta^2 f_\epsilon^\theta-\tilde\rho(n,\,p)\int_{B_{R}}f^{2}\eta^2 f_\epsilon^\theta\\
&+(\iota-1)\int_{B_{R}}f^{-1}|\nabla f|^2\eta^2 f_\epsilon^\theta.
\end{aligned}
$$
Hence
$$
\begin{aligned}
-2\int_{B_{R}} f_\epsilon^\theta  \eta|\nabla f||\nabla\eta|
\leq& 2(n-1)\kappa\int_{B_{R}}f\eta^2 f_\epsilon^\theta +\frac{2\left(n-2\right)}{n-1}\int_{B_{R}}|\nabla f| f^{\frac{1}{2}}\eta^2 f_\epsilon^\theta\\
&-\tilde\rho(n,\,p)\int_{B_{R}}f^{2}\eta^2 f_\epsilon^\theta+(\iota-1)\int_{B_{R}}f^{-1}|\nabla f|^2\eta^2 f_\epsilon^\theta\\
&-\theta\int_{B_{R}} f_\epsilon^{\theta-1} |\nabla f|^2\eta^2.
\end{aligned}
$$
By rearranging the above inequality, we have
$$
\begin{aligned}
&\theta\int_{B_{R}} f_\epsilon^{\theta-1} |\nabla f|^2\eta^2+\tilde\rho(n,\,p)\int_{B_{R}}f^{2}\eta^2 f_\epsilon^\theta\\
\leq& 2(n-1)\kappa\int_{B_{R}}f\eta^2 f_\epsilon^\theta+\frac{2\left(n-2\right)}{n-1}\int_{B_{R}}|\nabla f| f^{\frac{1}{2}}\eta^2 f_\epsilon^\theta+2\int_{B_{R}} f_\epsilon^\theta  \eta|\nabla f||\nabla\eta|\\
&+(\iota-1)\int_{B_{R}}f^{-1}|\nabla f|^2\eta^2 f_\epsilon^\theta\\
\leq& 2(n-1)\kappa\int_{B_{R}} f^{\theta+1}\eta^2 +\frac{2\left(n-2\right)}{n-1}\int_{B_{R}}|\nabla f| f^{\theta+\frac{1}{2}}\eta^2+2\int_{B_{R}} f^{\theta}\eta|\nabla f||\nabla\eta|\\
&+(\iota-1)\int_{B_{R}}f^{\theta-1}|\nabla f|^2\eta^2.
\end{aligned}
$$
By passing $\epsilon$ to $0$ we obtain
$$
\begin{aligned}
&\theta\int_{B_{R}} f^{\theta-1}|\nabla f|^2\eta^2+\tilde\rho(n,\,p)\int_{B_{R}}f^{\theta+2}\eta^2\\
\leq& 2(n-1)\kappa\int_{B_{R}} f^{\theta+1}\eta^2 +\frac{2\left(n-2\right)}{n-1}\int_{B_{R}}|\nabla f| f^{\theta+\frac{1}{2}}\eta^2+2\int_{B_{R}} f^{\theta}\eta|\nabla f||\nabla\eta|\\
&+(\iota-1)\int_{B_{R}}f^{\theta-1}|\nabla f|^2\eta^2,
\end{aligned}
$$
then, by rearranging the above we have
$$
\begin{aligned}
&(\theta+1-\iota)\int_{B_{R}} f^{\theta-1}|\nabla f|^2\eta^2+\tilde\rho(n,\,p)\int_{B_{R}}f^{\theta+2}\eta^2\\
\leq& 2(n-1)\kappa\int_{B_{R}} f^{\theta+1}\eta^2 +\frac{2\left(n-2\right)}{n-1}\int_{B_{R}}|\nabla f| f^{\theta+\frac{1}{2}}\eta^2+2\int_{B_{R}} f^{\theta}\eta|\nabla f||\nabla\eta|.
\end{aligned}
$$
Furthermore, by the choice of $\theta$ we know
\begin{equation}\label{eqno3.2}
\begin{aligned}
&\frac{\theta}{2}\int_{B_{R}} f^{\theta-1}|\nabla f|^2\eta^2+\tilde\rho(n,\,p)\int_{B_{R}}f^{\theta+2}\eta^2\\
\leq& 2(n-1)\kappa\int_{B_{R}} f^{\theta+1}\eta^2 +\frac{2\left(n-2\right)}{n-1}\int_{B_{R}}|\nabla f| f^{\theta+\frac{1}{2}}\eta^2+2\int_{B_{R}} f^{\theta}\eta|\nabla f||\nabla\eta|.
\end{aligned}
\end{equation}

On the other hand, by Young's inequality we can derive
$$
\begin{aligned}
2\int_{B_{R}} f^{\theta}\eta|\nabla f||\nabla\eta|
=&2\int_{B_{R}}f^{\frac{\theta-1}{2}}\eta\left|\nabla f\right|\times f^{\frac{\theta+1}{2}}\left|\nabla\eta\right|\\
\leq&2\int_{B_{R}}\left[\frac{\theta}{8}\frac{f^{\theta-1}\eta^2\left|\nabla f\right|^2}{2} + \frac{8}{\theta} \frac{f^{\theta+1}\left|\nabla\eta\right|^2}{2}\right]\\
\leq&\frac{\theta}{8}\int_{B_{R}}f^{\theta-1}\eta^2\left|\nabla f\right|^2+\frac{8}{\theta}\int_{B_{R}}f^{\theta+1}\left|\nabla\eta\right|^2,
\end{aligned}
$$
and
$$
\begin{aligned}
\frac{2\left(n-2\right)}{n-1}\int_{B_{R}}|\nabla f| f^{\theta+\frac{1}{2}}\eta^2
\leq&2\int_{B_{R}}f^{\theta+\frac{1}{2}}\eta^2\left|\nabla f\right|\\
\leq&2\int_{B_{R}}f^{\frac{\theta-1}{2}}\eta\left|\nabla f\right|\times f^{\frac{\theta+2}{2}}\eta\\
\leq&2\int_{B_{R}}\left[\frac{\theta}{8}\frac{f^{\theta-1}\eta^2\left|\nabla f\right|^2}{2} + \frac{8}{\theta} \frac{f^{\theta+2}\eta^2}{2}\right]\\
\leq&\frac{\theta}{8}\int_{B_{R}}f^{\theta-1}\eta^2\left|\nabla f\right|^2+\frac{8}{\theta}\int_{B_{R}}f^{\theta+2}\eta^2.
\end{aligned}
$$
Now, by picking $\theta$ such that
$$
\theta\geq\max\{\frac{16}{\tilde\rho(n,\,p)},\,2\iota\}>\max\{1,\,2(\iota-1)\}$$
which implies $$\frac{8}{\theta}\leq\frac{\tilde\rho(n,\,p)}{2},$$
in view of the above estimates of some terms on the right hand side of \eqref{eqno3.2} we can see easily from \eqref{eqno3.2} that there holds
\begin{equation}\label{eqno3.3}
\begin{aligned}
&\frac{\theta}{4}\int_{B_{R}} f^{\theta-1}|\nabla f|^2\eta^2+\frac{\tilde\rho(n,\,p)}{2}\int_{B_{R}}f^{\theta+2}\eta^2\\
\leq& 2(n-1)\kappa\int_{B_{R}} f^{\theta+1}\eta^2 +\frac{8}{\theta}\int_{B_{R}}f^{\theta+1}\left|\nabla\eta\right|^2.
\end{aligned}
\end{equation}

Besides, we have
$$
\begin{aligned}
\left|\nabla\left(f^{\frac{\theta+1 }{2}}\eta\right)\right|^2
=&\left|\eta\nabla f^{\frac{\theta+1}{2}}+f^{\frac{\theta+1}{2}}\nabla\eta\right|^2\\
\leq&2\eta^2\left|\nabla f^{\frac{\theta+1}{2}}\right|^2+2f^{\theta+1}\left|\nabla\eta\right|^2\\
=&\frac{\left(\theta+1\right)^2}{2}f^{\theta-1}\eta^2\left|\nabla f\right|^2+2f^{\theta+1}\left|\nabla\eta\right|^2,
\end{aligned}
$$
and integrate it on $B_{R}$ to obtain
$$
\begin{aligned}
\int_{B_{R}}\left|\nabla\left(f^{\frac{\theta+1 }{2}}\eta\right)\right|^2
\leq &\frac{\left(\theta+1\right)^2}{2}\int_{B_{R}}\eta^2f^{\theta-1 }\left|\nabla f\right|^2+2\int_{B_{R}}f^{\theta+1}\left|\nabla\eta\right|^2\\
\leq&\frac{2\left(\theta+1\right)^2}{\theta}\left[2(n-1)\kappa\int_{B_{R}} f^{\theta+1}\eta^2 +\frac{8}{\theta}\int_{B_{R}}f^{\theta+1}\left|\nabla\eta\right|^2 -\frac{\tilde\rho(n,\,p)}{2}\int_{B_{R}}f^{\theta+2}\eta^2\right]\\
& +2\int_{B_{R}}f^{\theta+1}\left|\nabla\eta\right|^2.
\end{aligned}
$$
Noticing that there holds true $$\theta^2<(\theta+1)^2\leq4\theta^2,$$
immediately we obtain from the inequality
$$
\begin{aligned}
\int_{B_{R}}\left|\nabla\left(f^{\frac{\theta+1 }{2}}\eta\right)\right|^2
\leq&8\theta\left[2(n-1)\kappa\int_{B_{R}} f^{\theta+1}\eta^2+ \frac{8}{\theta}\int_{B_{R}}f^{\theta+1}\left|\nabla\eta\right|^2-\frac{\tilde\rho(n,\,p)}{2}\int_{B_{R}}f^{\theta+2}\eta^2\right]\\
&+2\int_{B_{R}}f^{\theta+1}\left|\nabla\eta\right|^2\\
\leq&16(n-1)\kappa\theta\int_{B_{R}} f^{\theta+1}\eta^2 +66\int_{B_{R}}f^{\theta+1}\left|\nabla\eta\right|^2-4\theta\tilde\rho(n,\,p)\int_{B_{R}}f^{\theta+2}\eta^2.
\end{aligned}
$$
According to the Theorem \ref{thm Sobolev}, we deduce from the above inequality
$$
\begin{aligned}
\left(\int_{B_{R}}f^{(\theta+1)\lambda}\eta^{2\lambda}\right)^{\frac{1}{\lambda}}
\leq& e^{c_n\left(1+\sqrt \kappa R\right)}V^{-\frac{2}{n}}R^2\left[16(n-1)\kappa\theta\int_{B_{R}} f^{\theta+1}\eta^2 +66\int_{B_{R}}f^{\theta+1}\left|\nabla\eta\right|^2\right.\\
&\left.-4\theta\tilde\rho(n,\,p)\int_{B_{R}}f^{\theta+2}\eta^2+R^{-2}\int_{B_{R}}f^{\theta+1}\eta^2\right]\\
=& e^{c_n\left(1+\sqrt \kappa R\right)}V^{-\frac{2}{n}}\left[\left(16(n-1)\kappa\theta R^2+1\right)\int_{B_{R}} f^{\theta+1}\eta^2 \right.\\
&\left.+66R^2\int_{B_{R}}f^{\theta+1}\left|\nabla\eta\right|^2-4\theta\tilde\rho(n,\,p)R^2\int_{B_{R}}f^{\theta+2}\eta^2\right],
\end{aligned}
$$
where $V=~\mbox{Vol}~(B_{R})$ and $\lambda=\frac{n}{n-2}$. Rearranging the above inequality leads to the following
\begin{equation}\label{eqno3.4}
\begin{aligned}
&e^{-c_n\left(1+\sqrt \kappa R\right)}V^{\frac{2}{n}}\left(\int_{B_{R}}f^{(\theta+1)\lambda}\eta^{2\lambda}\right)^{\frac{1}{\lambda}}
+4\theta\tilde\rho(n,\,p)R^2\int_{B_{R}}f^{\theta+2}\eta^2\\
\leq&16(n-1)\theta \left(\kappa R^2+1\right)\int_{B_{R}} f^{\theta+1}\eta^2 +66R^2\int_{B_{R}}f^{\theta+1}\left|\nabla\eta\right|^2.
\end{aligned}
\end{equation}

Now we choose
$$
\theta_0=c_{n,\,p}\left(1+\sqrt\kappa R\right),
$$
where $$c_{n,\,p}=\max\{c_n,\,2\iota,\,\frac{16}{\tilde\rho(n,\,p)}\}.$$
Then, we can infer from \eqref{eqno3.4} that for any $\theta\geq\max\{2\iota,\,\frac{16}{\tilde\rho(n,\,p)}\}$ there holds
\begin{equation}\label{eqno3.5}
\begin{aligned}
&e^{-\theta_0}V^{\frac{2}{n}}\left(\int_{B_{R}}f^{(\theta+1)\lambda}\eta^{2\lambda}\right)^{\frac{1}{\lambda}}+
4\theta\tilde\rho(n,\,p)R^2\int_{B_{R}}f^{\theta+2}\eta^2\\
\leq&16(n-1)\theta \left(\kappa R^2+1\right)\int_{B_{R}} f^{\theta+1}\eta^2 +66R^2\int_{B_{R}}f^{\theta+1}\left|\nabla\eta\right|^2.
\end{aligned}
\end{equation}
By the definition of $\tilde\rho(n,\,p)$ in the previous section, it is easy to see that
$$
\frac{16}{\tilde\rho(n,\,p)}\geq8(n-1)\geq 8
$$
and
$$
16(n-1)\left(\kappa R^2+1\right)\leq\left[c_{n,\,p}\left(1+\sqrt\kappa R\right)\right]^2=\theta_0^2.
$$
In view of the above two facts we know that (\ref{eqno3.5}) can be rewritten as
\begin{equation}\label{eqno3.6}
\begin{aligned}
&e^{-\theta_0}V^{\frac{2}{n}}\left(\int_{B_{R}}f^{(\theta+1)\lambda}\eta^{2\lambda}\right)^{\frac{1}{\lambda}}+
4\theta\tilde\rho(n,\,p)R^2\int_{B_{R}}f^{\theta+2}\eta^2\\
\leq&\theta_0^2\theta \int_{B_{R}} f^{\theta+1}\eta^2 +66R^2\int_{B_{R}}f^{\theta+1}\left|\nabla\eta\right|^2,
\end{aligned}
\end{equation}
where $\theta\geq\max\{2\iota,\,\frac{16}{\tilde\rho(n,\,p)}\}$. Thus we complete the proof of this lemma.
\end{proof}

Using the above inequality we will infer a local estimate of $f$ stated in the following lemma, which will play a key role on the proofs of the main theorems.

\begin{lem}\label{lem3.1}
Let $\theta_1=(\theta_0+1)\lambda$ where $\theta_0=c_{n,\,p}\left(1+\sqrt\kappa R\right)$ defined as above. Then there exist a universal constant $c>0$ such that the following estimate of $\|f\|_{L^{\theta_1}\left(B_{3R/4}\right)}$ holds
\begin{equation}\label{eqno3.7}
\begin{aligned}
\|f\|_{L^{\theta_1}\left(B_{3R/4}\right)}
\leq& \frac{c}{\tilde\rho(n,\,p)} V^{\frac{1}{\theta_1}}\frac{\theta_0^2}{R^2},
\end{aligned}
\end{equation}
where $c$ is a universal constant and $\tilde\rho(n,\,p)$ is the same as above.
\end{lem}

\begin{proof}
Since the inequality (\ref{eqno3.6}) holds true for any $\theta\geq\theta_0$, letting $\theta=\theta_0$ in (\ref{eqno3.6}) we have
\begin{equation}\label{eqno3.8}
\begin{aligned}
&e^{-\theta_0}V^{\frac{2}{n}}\left(\int_{B_{R}}f^{(\theta_0+1)\lambda}\eta^{2\lambda}\right)^{\frac{1}{\lambda}}+
4\theta_0\tilde\rho(n,\,p)R^2\int_{B_{R}}f^{\theta_0+2}\eta^2\\
\leq&\theta_0^3 \int_{B_{R}} f^{\theta_0+1}\eta^2 +66R^2\int_{B_{R}}f^{\theta_0+1}\left|\nabla\eta\right|^2.
\end{aligned}
\end{equation}

For simplicity, we denote the first term on the $RHS$ of (\ref{eqno3.8}) by $R_{1}$ ($R_2,\,L_1,\,L_2$ are understood similarly).
Now, we focus on the $R_1$. Note that if
$$f\geq\frac{\theta_0^2}{2\tilde\rho(n,\,p)R^2},$$
then
$$R_1\leq 2\theta_0\tilde\rho(n,\,p)R^2\int_{B_{R}}f^{\theta_0+2}\eta^2=\frac{L_2}{2};$$
if
$$f<\frac{\theta_0^2}{2\tilde\rho(n,\,p)R^2},$$
then
$$R_1<\theta_0^3\left(\frac{\theta_0^2}{\tilde\rho(n,\,p)R^2}\right)^{\theta_0+1}V.$$
Therefore,
\begin{equation}\label{eqno3.9}
R_1\leq\frac{L_2}{2}+\theta_0^3\left(\frac{\theta_0^2}{\tilde\rho(n,\,p)R^2}\right)^{\theta_0+1}V.
\end{equation}

Next, we need to calculate the term $R_2$ by choosing some special test functions denoted by $\eta$. Choose $\eta_0\in C_0^{\infty}\left(B_R\right)$ such that
$$
\begin{cases}
0\leq\eta_0\leq1,&\mbox{on}~B_{R},\\
\eta_0=1,&\mbox{on}~B_{3R/4},\\
\left|\nabla\eta_0\right|\leq\frac{8}{R}. &
\end{cases}
$$
Let $\eta=\eta_0^{\theta_0+2}$. Then, direct computation yields
$$
\begin{aligned}
66R^2\left|\nabla\eta\right|^2
=&66R^2(\theta_0+2)^2\eta_0^{2(\theta_0+1)}|\nabla\eta_0|^2\\
\leq& 66\times4\theta_0^2\eta_0^{2(\theta_0+1)}\times8^2\\
=&16896\theta_0^2\eta^{\frac{2(\theta_0+1)}{\theta_0+2}}.
\end{aligned}
$$
This means that one can find a universal constant $c>16896$, which is independent of any parameter, such that
$$
R_2\leq c\theta_0^2\int_{B_{R}}f^{\theta_0+1}\eta^{\frac{2(\theta_0+1)}{\theta_0+2}}.
$$
By H\"{o}lder inequality, we have
$$
\begin{aligned}
c\theta_0^2\int_{B_{R}}f^{\theta_0+1}\eta^{\frac{2(\theta_0+1)}{\theta_0+2}}
\leq& c\theta_0^2\left( \int_{B_{R}}f^{\theta_0+2}\eta^2\right)^{\frac{\theta_0+1}{\theta_0+2}}\left(\int_{B_{R}}1\right)^{\frac{1}{\theta_0+2}}\\
=& c\theta_0^2\left( \int_{B_{R}}f^{\theta_0+2}\eta^2\right)^{\frac{\theta_0+1}{\theta_0+2}}V^{\frac{1}{\theta_0+2}}.
\end{aligned}
$$
Furthermore, for any $t>0$ we use Young's inequality to obtain
$$
\begin{aligned}
&c\theta_0^2\left( \int_{B_{R}}f^{\theta_0+2}\eta^2\right)^{\frac{\theta_0+1}{\theta_0+2}}V^{\frac{1}{\theta_0+2}}\\
=&\left( \int_{B_{R}}f^{\theta_0+2}\eta^2\right)^{\frac{\theta_0+1}{\theta_0+2}}t\times\frac{c\theta_0^2}{t}V^{\frac{1}{\theta_0+2}}\\
\leq&\frac{\theta_0+1}{\theta_0+2}\left[\left( \int_{B_{R}}f^{\theta_0+2}\eta^2\right)^{\frac{\theta_0+1}{\theta_0+2}}t\right]^\frac{\theta_0+2}{\theta_0+1}
+\frac{1}{\theta_0+2}\left(\frac{ c\theta_0^2}{t}V^{\frac{1}{\theta_0+2}}\right)^{\theta_0+2}\\
=&\frac{\theta_0+1}{\theta_0+2}t^\frac{\theta_0+2}{\theta_0+1} \int_{B_{R}}f^{\theta_0+2}\eta^2
+\frac{1}{\theta_0+2}t^{-(\theta_0+2)}\left(c\theta_0^2\right)^{\theta_0+2}V.
\end{aligned}
$$
Letting
$$
t=\left[\frac{2(\theta_0+2)\theta_0\tilde\rho(n,\,p)R^2}{(\theta+1)}\right]^{\frac{\theta_0+1}{\theta_0+2}},
$$
we can see that
$$
\frac{\theta_0+1}{\theta_0+2}t^\frac{\theta_0+2}{\theta_0+1} =2\theta_0\tilde\rho(n,\,p)R^2
$$
and
$$
\begin{aligned}
\frac{1}{\theta_0+2}t^{-(\theta_0+2)}
=&\frac{1}{\theta_0+2}\left[\frac{(\theta_0+1)}{2(\theta_0+2)\theta_0\tilde\rho(n,\,p)R^2}\right]^{\theta_0+1}\\
\leq&\left(\frac{1}{\theta_0\tilde\rho(n,\,p)R^2}\right)^{\theta_0+1}.
\end{aligned}
$$
Immediately, it follows
$$
\begin{aligned}
&c\theta_0^2\left( \int_{B_{R}}f^{\theta_0+2}\eta^2\right)^{\frac{\theta_0+1}{\theta_0+2}}V^{\frac{1}{\theta_0+2}}\\
\leq&2\theta_0\tilde\rho(n,\,p)R^2\int_{B_{R}}f^{\theta_0+2}\eta^2
+\left(\frac{1}{\theta_0\tilde\rho(n,\,p)R^2}\right)^{\theta_0+1}\left(c\theta_0^2\right)^{\theta_0+2}V\\
=&\frac{L_2}{2}+ c^{\theta_0+2}V\frac{\theta_0^2}{\theta_0^{\theta_0+1}}\left(\frac{\theta_0^2}{\tilde\rho(n,\,p)R^2}\right)^{\theta_0+1}.
\end{aligned}
$$
Hence, we obtain
\begin{equation}\label{eqno3.10}
R_2\leq \frac{L_2}{2}+ c^{\theta_0+2}V\left(\frac{\theta_0^2}{\tilde\rho(n,\,p)R^2}\right)^{\theta_0+1}.
\end{equation}
Substituting (\ref{eqno3.9}) and (\ref{eqno3.10}) into (\ref{eqno3.8}), we obtain
$$
\begin{aligned}
e^{-\theta_0}V^{\frac{2}{n}}\left(\int_{B_{R}}f^{(\theta_0+1)\lambda}\eta^{2\lambda}\right)^{\frac{1}{\lambda}}
\leq&\theta_0^3\left(\frac{\theta_0^2}{\tilde\rho(n,\,p)R^2}\right)^{\theta_0+1}V
+ c^{\theta_0+2}V\left(\frac{\theta_0^2}{\tilde\rho(n,\,p)R^2}\right)^{\theta_0+1}\\
=&V\left(\frac{\theta_0^2}{\tilde\rho(n,\,p)R^2}\right)^{\theta_0+1}\left(\theta_0^3+c^{\theta_0+2}\right),
\end{aligned}
$$
which implies
$$
\begin{aligned}
\left(\int_{B_{R}}f^{(\theta_0+1)\lambda}\eta^{2\lambda}\right)^{\frac{1}{\lambda}}
\leq&e^{\theta_0}V^{1-\frac{2}{n}}\left(\frac{\theta_0^2}{\tilde\rho(n,\,p)R^2}\right)^{\theta_0+1}\left(\theta_0^3+c^{\theta_0+2}\right).
\end{aligned}
$$
Thus, we arrive at
$$
\begin{aligned}
\|f\|_{L^{\theta_1}\left(B_{3R/4}\right)}
\leq&e^{\frac{\theta_0}{\theta_0+1}}V^{\frac{1}{\theta_1}}\frac{\theta_0^2}{\tilde\rho(n,\,p)R^2}\left(\theta_0^3+c^{\theta_0+2}\right)^{\frac{1}{\theta_0+1}}\\
\leq&eV^{\frac{1}{\theta_1}}\frac{\theta_0^2}{\tilde\rho(n,\,p)R^2}\left(\theta_0^{\frac{3}{\theta_0}}+c^2\right)\\
\leq&eV^{\frac{1}{\theta_1}}\frac{\theta_0^2}{\tilde\rho(n,\,p)R^2}c^2\left(\theta_0^{\frac{3}{\theta_0}}+1\right).
\end{aligned}
$$
Here we have used the fact for any two positive number $x$ and $y$ there holds true $$(x+y)^s\leq x^s + y^s$$ as $0<s<1$. Furthermore, by the properties of the function $y(t)=t^{\frac{3}{t}}$ on $(0,\,+\infty)$ we know that for any $\theta_0>0$
$$\theta_0^{\frac{3}{\theta_0}}+1\leq e^{\frac{3}{e}} + 1=\max_{t\in
(0, +\infty)}y(t).$$
Hence, (\ref{eqno3.7}) follows immediately. Thus, the proof of Lemma \ref{lem3.1} is completed.
\end{proof}

Now, we are in the position to give the proof of Theorem \ref{thm1} by applying the Nash-Moser iteration method.
\begin{proof} Assume $v$ is a smooth positive solution of (\ref{eqno2.1}). Since (\ref{eqno1.1}) is defined on a complete Riemannian manifold $(M,\,g)$ with Ricci curvature $Ric(M)\geq-(n-1)\kappa$,  $a$ and $p$ satisfy one of the following two conditions:
\begin{enumerate}
\item  $a\left(\frac{2}{n-1}-p\right)\geq0$ and $p\geq -1$;
\item  $0< p< \frac{4}{n-1}$,
\end{enumerate}
by the above arguments on
$$
f=\left|\nabla u\right|^2,
$$
where $u=-\ln v$, now we go back to (\ref{eqno3.6}) and ignore the second term on its $LHS$ to obtain
$$
\begin{aligned}
e^{-\theta_0}V^{\frac{2}{n}}\left(\int_{B_{R}}f^{(\theta+1)\lambda}\eta^{2\lambda}\right)^{\frac{1}{\lambda}}
\leq& \theta_0^2\theta \int_{B_{R}} f^{\theta+1}\eta^2 +66R^2\int_{B_{R}}f^{\theta+1}\left|\nabla\eta\right|^2\\
\leq&66\int_{B_{R}} f^{\theta+1}\left(\theta_0^2\theta \eta^2 +R^2\left|\nabla\eta\right|^2\right),
\end{aligned}
$$
which is equivalent to
\begin{equation}\label{eqno3.11}
\left(\int_{B_{R}}f^{(\theta+1)\lambda}\eta^{2\lambda}\right)^{\frac{1}{\lambda}}
\leq 66e^{\theta_0}V^{-\frac{2}{n}}\int_{B_{R}} f^{\theta+1}\left(\theta_0^2\theta\eta^2  +R^2\left|\nabla\eta\right|^2\right).
\end{equation}

In consideration of the delicate requirements of $\theta$, we take an increasing sequence $\{\theta_k\}_{k=1}^{\infty}$ such that
$$
\theta_1=(\theta_0+1)\lambda\quad\mbox{and}\quad\theta_{k+1}=\theta_{k}\lambda,\quad k=1,\,2,\,\cdots,
$$
and a decreasing one $\{r_k\}_{k=1}^{\infty}$ such that
$$
r_k=\frac{R}{2}+\frac{R}{4^k},\quad k=1,\,2,\,\cdots\,.
$$
Then, we may choose $\{\eta_k\}_{k=1}^{\infty}\subset C_0^{\infty}(B_R)$, such that $\eta_k\in C_0^{\infty}(B_{r_k})$,
$$
\eta_k=1 \quad\mbox{in}~B_{r_{k+1}}\quad\mbox{and}\quad\left|\nabla\eta_k\right|\leq\frac{4^{k+1}}{R}.
$$
By letting $\theta+1=\theta_k$ and $\eta=\eta_k$ in (\ref{eqno3.11}), we can derive
$$
\begin{aligned}
\left(\int_{B_{R}}f^{\theta_k\lambda}
\eta_k^{2\lambda}\right)^{\frac{1}{\lambda}}
\leq& ce^{\theta_0}V^{-\frac{2}{n}}\int_{B_{R}} f^{\theta_k}\left[\theta_0^2\theta_k\eta_k^2  +R^2\left|\nabla\eta_k\right|^2\right]\\
\leq& ce^{\theta_0}V^{-\frac{2}{n}}\int_{B_{R}} f^{\theta_k}\left[\theta_0^2\theta_k\eta_k^2  +R^2\left(\frac{4^{k+1}}{R}\right)^2\right]\\
\leq& ce^{\theta_0}V^{-\frac{2}{n}}\left(\theta_0^2\theta_k  +16^{k+1}\right)\int_{B_{r_{k}}} f^{\theta_k}\\
\leq& ce^{\theta_0}V^{-\frac{2}{n}}\left[\theta_0^2(\theta_0+1)\lambda^k  +16^{k+1}\right]\int_{B_{r_{k}}} f^{\theta_k}\\
\leq& c e^{\theta_0}V^{-\frac{2}{n}}\left(\theta_0^3 16^k  +16^{k}\right)\int_{B_{r_{k}}} f^{\theta_k}\\
\leq& c e^{\theta_0}V^{-\frac{2}{n}}\theta_0^3 16^k \int_{B_{r_{k}}} f^{\theta_k}.
\end{aligned}
$$
Thus,
$$
\begin{aligned}
\left(\int_{B_{r_{k+1}}}f^{\theta_{k+1}}\right)^{\frac{1}{\theta_{k+1}}}
\leq& \left(c e^{\theta_0}V^{-\frac{2}{n}}\theta_0^3\right)^{\frac{1}{\theta_k}} 16^{\frac{k}{\theta_k}}\left( \int_{B_{r_{k}}} f^{\theta_k}\right)^{\frac{1}{\theta_k}},\\
\end{aligned}
$$
where $c$ is a universal positive constant which does not depend on any parameter. This means that
$$
\|f\|_{L^{\theta_{k+1}}\left(B_{r_{k+1}}\right)}\leq
\left(ce^{\theta_0}V^{-\frac{2}{n}}\theta_0^3\right)^{\frac{1}{\theta_{k}}}16^{\frac{k}{\theta_k}} \|f\|_{L^{\theta_{k}}\left(B_{r_{k}}\right)}.
$$
By iteration we have
\begin{equation}\label{eqno3.12}
\|f\|_{L^{\theta_{k+1}}\left(B_{r_{k+1}}\right)}\leq
\left(ce^{\theta_0}V^{-\frac{2}{n}}\theta_0^3\right)^{\sum_{i=1}^{k}\frac{1}{\theta_{i}}}16^{\sum_{i=1}^{k}\frac{i}{\theta_i}} \|f\|_{L^{\theta_{1}}\left(B_{3R/4}\right)}.
\end{equation}
In view of
$$
\begin{aligned}
\sum_{i=1}^{\infty}\frac{1}{\theta_{i}}
=&\frac{1}{\theta_0+1}\sum_{i=1}^{\infty}\frac{1}{\lambda^i}\\
=&\frac{1}{\theta_0+1}\lim\limits_{i\to+\infty}\frac{\frac{1}{\lambda}(1-\frac{1}{\lambda^i})}{1-\frac{1}{\lambda}}\\
=&\frac{n-2}{\theta_0+1}\lim\limits_{i\to+\infty}(1-\frac{1}{\lambda^i})\\
=&\frac{n-2}{\theta_0+1}\\
=&\frac{n}{2\theta_1}
\end{aligned}
$$
and
$$
\begin{aligned}
\sum_{i=1}^{\infty}\frac{i}{\theta_{i}}
=&\frac{1}{\theta_0+1}\sum_{i=1}^{\infty}\frac{i}{\lambda^i}\\
=&\frac{1}{\theta_0+1}\frac{1}{\lambda-1}\sum_{i=1}^{\infty}\left[(\lambda-1)\frac{i}{\lambda^i}\right]\\
=&\frac{1}{\theta_0+1}\frac{n-2}{2}\sum_{i=1}^{\infty}\left(\frac{i}{\lambda^{i-1}}-\frac{i}{\lambda^{i}}\right)\\
=&\frac{n}{2\theta_1}\left[1+\lim\limits_{i\to+\infty}\left(\frac{1}{\lambda}\frac{1-\frac{1}{\lambda^{i-1}}}{1-\frac{1}{\lambda}}-\frac{i}{\lambda^{i}}\right)\right]\\
=&\frac{n}{2\theta_1}\left(1+\frac{1}{\lambda-1}\right)\\
=&\frac{n^2}{4\theta_1},
\end{aligned}
$$
by letting $k\rightarrow\infty$ in (\ref{eqno3.12}) we obtain the following
$$
\|f\|_{L^{\infty}\left(B_{R/2}\right)}\leq c(n)V^{-\frac{1}{\theta_1}} \|f\|_{L^{\theta_{1}}\left(B_{3R/4}\right)}.
$$
By Lemma \ref{lem3.1}, we conclude from the above inequality that
$$
\|f\|_{L^{\infty}\left(B_{R/2}\right)}\leq c(n) \frac{\theta_0^2}{\tilde\rho(n,\,p)R^2}.
$$
The definition of $\theta_0$ tells us that it follows
$$
\begin{aligned}
\|f\|_{L^{\infty}\left(B_{R/2}\right)}
\leq& c\left(n,\,p,\,\tilde\rho(n,\,p)\right)\frac{\left(1+\sqrt\kappa R\right)^2}{R^2}\\
=& c\left(n,\,p\right)\frac{\left(1+\sqrt\kappa R\right)^2}{R^2}.
\end{aligned}
$$
Thus, we complete the proof of Theorem \ref{thm1}.
\end{proof}

Now, we turn to proving Corollary \ref{cor1}.

\begin{proof}
Let $(M,\,g)$ be a noncompact complete Riemannian manifold with nonnegative Ricci curvature.  We assume $v$ is a smooth and positive solution of (\ref{eqno1.1}) with $\mu\geq0$, $a\in\mathbb{R}$, $b\geq0$, $p\geq-1$ and $q\geq1$.
If $a$ and $p$ satisfy one of the following conditions:
\begin{enumerate}
\item  $a\left(\frac{2}{n-1}-p\right)\geq0$ and $p\geq -1$;
\item  $0< p< \frac{4}{n-1}$,
\end{enumerate}
then Theorem \ref{thm1} tells us that there holds for any $B_{R}\subset M$,
$$
\frac{\left|\nabla v\right|^2}{v^2}\leq\frac{c(n,\,p)}{R^2},\quad\mbox{on}~B_{R/2}.
$$
Letting $R\rightarrow\infty$ yields $\nabla v=0$. Therefore, $v$ is a positive constant on $M$.

Furthermore, if one of the following conditions holds true:
\begin{enumerate}
\item $a>0$ and $-1\leq p <\frac{4}{n-1}$;
\item $a=0$, $\mu+b\neq0$, $p\geq -1$;
\item $a<0$, $b=\mu=0$ and $p>0$.
\end{enumerate}
then we have
\begin{equation}\label{neq}
\Delta v +\mu v+av^{p+1}+bv^{-q+1}\neq0,
\end{equation}
this is a contradiction which means that $v$ could not be the solution to (\ref{eqno1.1}). Hence we know that (\ref{eqno1.1}) does not admit any positive solution.

On the other hands, if $a<0$, $p>0$ and at least one of $b$ and $\mu$ is positive, we do not know whether (\ref{neq}) is true or not. Therefore, we can only infer that in this case $v$ is a positive constant on $M$. Thus we complete the proof of Corollary \ref{cor1}.
\end{proof}

\section{Proof of Theorem \ref{thm2}}
Assume $v$ is a smooth positive solution of \eqref{eqno1.1}. Now, we turn to discussing \eqref{eqno1.1} in the case the coefficients $\mu<0$, $a>0$ and $b>0$ which are real constant throughout this section, unless otherwise mentioned.

\begin{lem}\label{lem4.1}
Let $a>0$, $b>0$ and $\max\{-a,\,-b\}<\mu<0$. Assume that $f=\left|\nabla u\right|^2$ where $u=-\ln v$ and $v$ is a smooth positive solution of (\ref{eqno1.1}). Then, for any $q\geq1$, there exist $\iota\geq1$ and a positive constant $\alpha\left(n,\,p,\,\delta\right)$, which depend on $n$, $p$ and $\delta$, such that
\begin{equation}\label{eqno4.2}
\frac{\Delta \left(f^\iota\right)}{\iota f^{\iota-1}}
\geq-2(n-1)\kappa f+\frac{2(n-2)\langle\nabla u,\,\nabla f\rangle}{n-1}+\alpha\left(n,\,p,\,\delta\right)f^2,
\end{equation}
if $p$ satisfies $$0<p<\frac{4(1-\delta)}{n-1}$$
where $\delta=\frac{\mu}{\max\{-a,\,-b\}}$.
\end{lem}

\begin{proof}
From the previous arguments on $u$, we have
$$
\Delta u=f+\mu+ae^{-pu}+be^{qu},
$$
where $f=\left|\nabla u\right|^2$. Since $b\geq0$, by Bochner formula we have that, for any $q\geq1$ and $p\geq-1$, at the point $f\neq0$ there holds
$$
\Delta f \geq 2|\nabla ^2 u|^2 -2(n-1)\kappa f+ 2\langle \nabla u,\,\nabla f\rangle - 2fp(ae^{-pu}+be^{qu}).
$$

We choose a suitable local orthonormal frame $\left\{\xi_i \right\}_{i=1}^{n}$ by the same way as we chose it in Section 2. Now, by performing a similar argument with the previous section we can infer that, for any $\iota\geq1$ there holds
\begin{equation}\label{eqno4.1}
\begin{aligned}
\frac{\Delta \left(f^\iota\right)}{\iota f^{\iota-1}}
\geq&-2(n-1)\kappa f+\frac{2f^2}{n-1}+\frac{2(n-2)\langle\nabla u,\,\nabla f\rangle}{n-1}\\
&+\frac{2(2\iota-1)}{2(\iota-1)(n-1)+n}\left(\mu+ae^{-pu}+be^{qu}\right)^2\\
&+4f\frac{\mu+ae^{-pu}+be^{qu}}{n-1}-2pf \left(ae^{-pu}+ be^{qu} \right).
\end{aligned}
\end{equation}

If $a>0$, $b>0$ and $p\geq0$, then we have that for any $u\in\mathbb{R}$,
$$
ae^{-pu}+be^{qu}\geq\min\{a,\,b\}.
$$
For any $\mu<0$, we set
$$
\delta=\frac{\mu}{\max \{-a,\,-b\}}>0.
$$
Then, obviously
$$
\mu+ae^{-pu}+be^{qu}=ae^{-pu}+be^{qu}-\delta\min\{a,\,b\}\geq(1-\delta)(ae^{-pu}+be^{qu}).
$$
Moreover, since $\mu$ satisfies $$\max\{-a,\,-b\}<\mu<0,$$
obviously, we have $\delta\in(0,\,1)$. Hence, \eqref{eqno4.1} can be rewritten as
$$
\begin{aligned}
\frac{\Delta \left(f^\iota\right)}{\iota f^{\iota-1}}
\geq &-2(n-1)\kappa f+\frac{2f^2}{n-1}+\frac{2(n-2)\langle\nabla u,\,\nabla f\rangle}{n-1}\\
&+\frac{2(2\iota-1)}{2(\iota-1)(n-1)+n}(1-\delta)^2\left(ae^{-pu}+be^{qu}\right)^2\\
&+2f(1-\delta)(ae^{-pu}+be^{qu})\left(\frac{2}{n-1}-\frac {p}{1-\delta} \right).
\end{aligned}
$$
For the sake of convenience, we denote
$$
\tau=(1-\delta)(ae^{-pu}+be^{qu}).
$$
Then, the above inequality reads
$$
\begin{aligned}
\frac{\Delta \left(f^\iota\right)}{\iota f^{\iota-1}}
\geq&-2(n-1)\kappa f+\frac{2f^2}{n-1}+\frac{2(n-2)\langle\nabla u,\,\nabla f\rangle}{n-1}\\
&+\frac{2(2\iota-1)}{2(\iota-1)(n-1)+n}\tau^2+2f\tau\left(\frac{2}{n-1}-\frac{p}{1-\delta}\right).
\end{aligned}
$$
Since
$$
\begin{aligned}
&\frac{2(2\iota-1)}{2(\iota-1)(n-1)+n}\tau^2+2f\tau\left(\frac{2}{n-1}-\frac{p}{1-\delta}\right)\\
\geq&-\frac{2(\iota-1)(n-1)+n}{2(2\iota-1)}\left(\frac{2}{n-1}-\frac{p}{1-\delta}\right)^2f^2,
\end{aligned}
$$
we can see that
\begin{equation}\label{eqno4.3}
\begin{aligned}
\frac{\Delta \left(f^\iota\right)}{\iota f^{\iota-1}}
\geq&-2(n-1)\kappa f+\frac{2(n-2)\langle\nabla u,\,\nabla f\rangle}{n-1}\\
&+\left[\frac{2}{n-1}-\frac{2(\iota-1)(n-1)+n}{2(2\iota-1)}\left(\frac{2}{n-1}-\frac{p}{1-\delta}\right)^2\right]f^2.
\end{aligned}
\end{equation}
Obviously, when
$$0<p<\frac{4(1-\delta)}{n-1},$$
there holds
$$0\leq\left|\frac{2}{n-1}-\frac{p}{1-\delta}\right|<\frac{2}{n-1}.$$
According to the properties of the function $$y(\iota)=\frac{2(\iota-1)(n-1)+n}{2(2\iota-1)},$$
we can choose $\iota=\iota(n,\,p,\,\delta)$ large enough such that
$$
\frac{2(\iota-1)(n-1)+n}{2(2\iota-1)}\left(\frac{2}{n-1}-\frac{p}{1-\delta}\right)^2<\frac{2}{n-1}.
$$
Set
$$\alpha\left(n,\,p,\,\delta\right)=\frac{2}{n-1}-\frac{2(\iota-1)(n-1)+n}{2(2\iota-1)}\left(\frac{2}{n-1}-\frac{p}{1-\delta}\right)^2.$$
Here, $\iota=\iota(n,\,p,\,\delta)$ is chosen in the above. Obviouly,
$$\alpha\left(n,\,p,\,\delta\right)\in\left(0,\,\frac{2}{n-1}\right].$$
Thus, we get (\ref{eqno4.2}) and finish the proof of  Lemma \ref{lem4.1}.\end{proof}

\begin{lem}
Let $a>0$, $b>0$ and $\max\{-a,\,-b\}<\mu<0$. Assume that $f=\left|\nabla u\right|^2$ where $u=-\ln v$ and $v$ is a smooth positive solution of (\ref{eqno1.1}). Then, there exists $\gamma_0=c_{n,p}(1+\sqrt{\kappa}R)$, where $c_{n, p}=\max\{c_n,\,2\iota,\,\frac{16}{\alpha(n,\,p,\,\delta)}\}$ is a positive constant depending on $n$, $p$ and $\alpha(n,\,p,\,\delta)$ which is defined as above, such that for any $0\leq\eta\in C_{0}^{\infty}(B_{R})$ and any $\gamma\geq\gamma_0$ large enough there holds true
\begin{equation}\label{eqno4.9}
\begin{aligned}
&e^{-\gamma_0}V^{\frac{2}{n}}\left(\int_{B_{R}}f^{(\gamma+1)\lambda}\eta^{2\lambda}\right)^{\frac{1}{\lambda}}+
4\gamma\alpha(n,\,p,\,\delta)R^2\int_{B_{R}}f^{\gamma+2}\eta^2\\
\leq&\gamma_0^2\gamma \int_{B_{R}} f^{\gamma+1}\eta^2 +66R^2\int_{B_{R}}f^{\gamma+1}\left|\nabla\eta\right|^2.
\end{aligned}
\end{equation}
Here $B_R$ is a geodesic Ball in $(M, g)$ and $V$ is the volume of $B_R$.
\end{lem}

\begin{proof}
By the approximation used in Section 3, we may assume that the function $f>0$ on $B_{R}$ without loss of generality.
According to the Lemma \ref{lem4.1}, we can see that for any $q\geq 1$ and any function $\varphi\in W_{0}^{1,2}(B_R)$ with $\varphi\geq 0$, there holds true
$$
\int_{B_{R}}\frac{\Delta \left(f^\iota\right)}{\iota f^{\iota-1}}\varphi
\geq-2(n-1)\kappa \int_{B_{R}}f\varphi +\frac{2\left(n-2\right)}{n-1}\int_{B_{R}}\langle\nabla u,\,\nabla f\rangle \varphi+\alpha\left(n,\,p,\,\delta\right)\int_{B_{R}}f^{2}\varphi,
$$
if $\iota=\iota(n,\,p,\,\delta)$ is chosen as above.

Substituting (\ref{eqno2.6}) into the above, we arrive at
\begin{equation}\label{eqno4.4}
\begin{aligned}
\int_{B_{R}}\langle \nabla f,\, \nabla\varphi \rangle\
\leq&2(n-1)\kappa \int_{B_{R}}f\varphi -\frac{2\left(n-2\right)}{n-1}\int_{B_{R}}\langle\nabla u,\,\nabla f\rangle \varphi\\
&-\alpha\left(n,\,p,\,\delta\right)\int_{B_{R}}f^{2}\varphi+(\iota-1)\int_{B_{R}}f^{-1}|\nabla f|^2\varphi.
\end{aligned}
\end{equation}

Now, let $$\varphi=\eta^2 f^\gamma\in W_{0}^{1,2}(B_R),$$
where $\eta\in C_{0}^{\infty}(B_{R}),\,\eta\geq0$ and $\gamma>\max\{1,\,2(\iota-1)\}$ will be determined later.
Hence,
$$
\begin{aligned}
&\int_{B_{R}} \langle\nabla f,\,2f^\gamma\eta\nabla\eta+\gamma f^{\gamma-1}\eta^2\nabla f\rangle\\
\leq& 2(n-1)\kappa\int_{B_{R}}f^{\gamma+1}\eta^2 -\frac{2\left(n-2\right)}{n-1}\int_{B_{R}}\langle\nabla u,\,\nabla f\rangle f^{\gamma}\eta^2\\
&-\alpha\left(n,\,p,\,\delta\right)\int_{B_{R}}f^{\gamma+2}\eta^2+(\iota-1)\int_{B_{R}}f^{\gamma-1}|\nabla f|^2\eta^2 ,
\end{aligned}
$$
which implies
$$
\begin{aligned}
&2\int_{B_{R}} f^{\gamma}\eta\langle\nabla f,\,\nabla\eta\rangle+\gamma\int_{B_{R}} f^{\gamma-1}|\nabla f|^2\eta^2\\
\leq& 2(n-1)\kappa\int_{B_{R}}f^{\gamma+1}\eta^2 -\frac{2\left(n-2\right)}{n-1}\int_{B_{R}}\langle\nabla u,\,\nabla f\rangle f^{\gamma}\eta^2 \\
&-\alpha\left(n,\,p,\,\delta\right)\int_{B_{R}}f^{\gamma+2}\eta^2+(\iota-1)\int_{B_{R}}f^{\gamma-1}|\nabla f|^2\eta^2.
\end{aligned}
$$
It is not difficult to see that
$$
\begin{aligned}
&-2\int_{B_{R}} f^{\gamma}\eta|\nabla f||\nabla\eta|+(\gamma+1-\iota)\int_{B_{R}} f^{\gamma-1}|\nabla f|^2\eta^2\\
\leq & 2(n-1)\kappa\int_{B_{R}}f^{\gamma+1}\eta^2 +\frac{2\left(n-2\right)}{n-1}\int_{B_{R}}|\nabla f| f^{\gamma+\frac{1}{2}}\eta^2-\alpha\left(n,\,p,\,\delta\right)\int_{B_{R}}f^{\gamma+2}\eta^2.
\end{aligned}
$$
By rearranging the above inequality, we have
$$
\begin{aligned}
&(\gamma+1-\iota)\int_{B_{R}} f^{\gamma-1}|\nabla f|^2\eta^2+\alpha\left(n,\,p,\,\delta\right)\int_{B_{R}}f^{\gamma+2}\eta^2\\
\leq& 2(n-1)\kappa\int_{B_{R}}f^{\gamma+1}\eta^2 +\frac{2\left(n-2\right)}{n-1}\int_{B_{R}}|\nabla f| f^{\gamma+\frac{1}{2}}\eta^2+2\int_{B_{R}} f^{\gamma}\eta|\nabla f||\nabla\eta|,
\end{aligned}
$$
and it follows
\begin{equation}\label{eqno4.5}
\begin{aligned}
&\frac{\gamma}{2}\int_{B_{R}} f^{\gamma-1}|\nabla f|^2\eta^2+\alpha\left(n,\,p,\,\delta\right)\int_{B_{R}}f^{\gamma+2}\eta^2\\
\leq& 2(n-1)\kappa\int_{B_{R}}f^{\gamma+1}\eta^2 +\frac{2\left(n-2\right)}{n-1}\int_{B_{R}}|\nabla f| f^{\gamma+\frac{1}{2}}\eta^2+2\int_{B_{R}} f^{\gamma}\eta|\nabla f||\nabla\eta|.
\end{aligned}
\end{equation}

By Young's inequality, we can derive that the third term on the right hand side of \eqref{eqno4.5} satisfies
$$
\begin{aligned}
2\int_{B_{R}} f^{\gamma}\eta|\nabla f||\nabla\eta|
\leq&\frac{\gamma}{8}\int_{B_{R}}f^{\gamma-1}\eta^2\left|\nabla f\right|^2+\frac{8}{\gamma}\int_{B_{R}}f^{\gamma+1}\left|\nabla\eta\right|^2,
\end{aligned}
$$
and the second term on the right hand side of \eqref{eqno4.5} satisfies
$$
\begin{aligned}
\frac{2\left(n-2\right)}{n-1}\int_{B_{R}}|\nabla f| f^{\gamma+\frac{1}{2}}\eta^2
\leq&\frac{\gamma}{8}\int_{B_{R}}f^{\gamma-1}\eta^2\left|\nabla f\right|^2+\frac{8}{\gamma}\int_{B_{R}}f^{\gamma+2}\eta^2.
\end{aligned}
$$
By picking $\gamma$ such that
$$
\gamma\geq\max\{\frac{16}{\alpha(n,\,p,\,\delta)},\,2\iota\}>\max\{1,\,2(\iota-1)\}$$
which implies
$$\frac{8}{\gamma}\leq\frac{\alpha\left(n,\,p,\,\delta\right)}{2},$$
we can deduce from (\ref{eqno4.5}) that there holds true
\begin{equation}\label{eqno4.6}
\begin{aligned}
&\frac{\gamma}{4}\int_{B_{R}} f^{\gamma-1}|\nabla f|^2\eta^2+\frac{\alpha\left(n,\,p,\,\delta\right)}{2}\int_{B_{R}}f^{\gamma+2}\eta^2\\
\leq& 2(n-1)\kappa\int_{B_{R}} f^{\gamma+1}\eta^2 +\frac{8}{\gamma}\int_{B_{R}}f^{\gamma+1}\left|\nabla\eta\right|^2.
\end{aligned}
\end{equation}

Besides, we have
$$
\begin{aligned}
\int_{B_{R}}\left|\nabla\left(f^{\frac{\gamma+1 }{2}}\eta\right)\right|^2
\leq &\frac{\left(\gamma+1\right)^2}{2}\int_{B_{R}}\eta^2f^{\gamma-1 }\left|\nabla f\right|^2+2\int_{B_{R}}f^{\gamma+1}\left|\nabla\eta\right|^2\\
\leq&16(n-1)\kappa\gamma\int_{B_{R}} f^{\gamma+1}\eta^2 +66\int_{B_{R}}f^{\gamma+1}\left|\nabla\eta\right|^2-4\gamma\alpha(n,\,p,\,\delta)\int_{B_{R}}f^{\gamma+2}\eta^2.
\end{aligned}
$$
According to the Theorem \ref{thm Sobolev}, we obtain
$$
\begin{aligned}
\left(\int_{B_{R}}f^{(\gamma+1)\lambda}\eta^{2\lambda}\right)^{\frac{1}{\lambda}}
\leq& e^{c_n\left(1+\sqrt \kappa R\right)}V^{-\frac{2}{n}}R^2\left[16(n-1)\kappa\gamma\int_{B_{R}} f^{\gamma+1}\eta^2 +66\int_{B_{R}}f^{\gamma+1}\left|\nabla\eta\right|^2\right.\\
&\left.-4\gamma\alpha(n,\,p,\,\delta)\int_{B_{R}}f^{\gamma+2}\eta^2+R^{-2}\int_{B_{R}}f^{\gamma+1}\eta^2\right]\\
=& e^{c_n\left(1+\sqrt \kappa R\right)}V^{-\frac{2}{n}}\left[\left(16(n-1)\kappa\gamma R^2+1\right)\int_{B_{R}} f^{\gamma+1}\eta^2 \right.\\
&\left.+66R^2\int_{B_{R}}f^{\gamma+1}\left|\nabla\eta\right|^2-4\gamma\alpha(n,\,p,\,\delta)R^2\int_{B_{R}}f^{\gamma+2}\eta^2\right],
\end{aligned}
$$
where $V=~\mbox{Vol}~(B_{R})$ and $\lambda=\frac{n}{n-2}$.
Hence, it follows from the above inequality
\begin{equation}\label{eqno4.7}
\begin{aligned}
&e^{-c_n\left(1+\sqrt \kappa R\right)}V^{\frac{2}{n}}\left(\int_{B_{R}}f^{(\gamma+1)\lambda}\eta^{2\lambda}\right)^{\frac{1}{\lambda}}
+4\gamma\alpha(n,\,p,\,\delta)R^2\int_{B_{R}}f^{\gamma+2}\eta^2\\
\leq&16(n-1)\gamma \left(\kappa R^2+1\right)\int_{B_{R}} f^{\gamma+1}\eta^2 +66R^2\int_{B_{R}}f^{\gamma+1}\left|\nabla\eta\right|^2.
\end{aligned}
\end{equation}

Now we choose
$$
\gamma_0=c_{n,\,p}\left(1+\sqrt\kappa R\right),
$$
where $$c_{n,\,p}=\max\{c_n,\,2\iota,\,\frac{16}{\alpha(n,\,p,\,\delta)}\}.
$$
It is easy to see that \eqref{eqno4.7} implies
\begin{equation}\label{eqno4.8}
\begin{aligned}
&e^{-\gamma_0}V^{\frac{2}{n}}\left(\int_{B_{R}}f^{(\gamma+1)\lambda}\eta^{2\lambda}\right)^{\frac{1}{\lambda}}+
4\gamma\alpha(n,\,p,\,\delta)R^2\int_{B_{R}}f^{\gamma+2}\eta^2\\
\leq&16(n-1)\gamma \left(\kappa R^2+1\right)\int_{B_{R}} f^{\gamma+1}\eta^2 +66R^2\int_{B_{R}}f^{\gamma+1}\left|\nabla\eta\right|^2.
\end{aligned}
\end{equation}
Noting that
$$
\frac{16}{\alpha(n,\,p,\,\delta)}\geq8(n-1)\geq 8
$$
and
$$
16(n-1)\left(\kappa R^2+1\right)\leq\left[c_{n,\,p}\left(1+\sqrt\kappa R\right)\right]^2=\gamma_0^2,
$$
from (\ref{eqno4.8}) we infer the following
\begin{equation}\label{eqno4.9}
\begin{aligned}
&e^{-\gamma_0}V^{\frac{2}{n}}\left(\int_{B_{R}}f^{(\gamma+1)\lambda}\eta^{2\lambda}\right)^{\frac{1}{\lambda}}+
4\gamma\alpha(n,\,p,\,\delta)R^2\int_{B_{R}}f^{\gamma+2}\eta^2\\
\leq&\gamma_0^2\gamma \int_{B_{R}} f^{\gamma+1}\eta^2 +66R^2\int_{B_{R}}f^{\gamma+1}\left|\nabla\eta\right|^2.
\end{aligned}
\end{equation}
Thus, we finish the proof of this lemma.
\end{proof}

The above inequality will lead to a local estimate of $f$ stated in the following lemma, which will play a key role on the following arguments.

\begin{lem}\label{lem4.2}
Let $\gamma_1=(\gamma_0+1)\lambda$. Then there exist a universal constant $c>0$ such that the following estimate holds
\begin{equation}\label{eqno4.10}
\begin{aligned}
\|f\|_{L^{\gamma_1}\left(B_{3R/4}\right)}
\leq& \frac{c}{\alpha(n,\,p,\,\delta)} V^{\frac{1}{\gamma_1}}\frac{\gamma_0^2}{R^2}.
\end{aligned}
\end{equation}

\end{lem}

\begin{proof}
Letting $\gamma=\gamma_0$ in (\ref{eqno4.9}), we can derive
\begin{equation}\label{eqno4.11}
\begin{aligned}
&e^{-\gamma_0}V^{\frac{2}{n}}\left(\int_{B_{R}}f^{(\gamma_0+1)\lambda}\eta^{2\lambda}\right)^{\frac{1}{\lambda}}+
4\gamma_0\alpha(n,\,p,\,\delta)R^2\int_{B_{R}}f^{\gamma_0+2}\eta^2\\
\leq&\gamma_0^3 \int_{B_{R}} f^{\gamma_0+1}\eta^2 +66R^2\int_{B_{R}}f^{\gamma_0+1}\left|\nabla\eta\right|^2.
\end{aligned}
\end{equation}

Note that if
$$f\geq\frac{\gamma_0^2}{2\alpha(n,\,p,\,\delta)R^2},$$
then
$$\gamma_0^3 \int_{B_{R}} f^{\gamma_0+1}\eta^2\leq 2\gamma_0\alpha(n,\,p,\,\delta)R^2\int_{B_{R}}f^{\gamma_0+2}\eta^2.$$
On the other hand, if
$$f<\frac{\gamma_0^2}{2\alpha(n,\,p,\,\delta)R^2},$$
then
$$\gamma_0^3 \int_{B_{R}} f^{\gamma_0+1}\eta^2<\gamma_0^3\left(\frac{\gamma_0^2}{\alpha(n,\,p,\,\delta)R^2}\right)^{\gamma_0+1}V.$$
Therefore, there always holds
\begin{equation}\label{eqno4.12}
\gamma_0^3 \int_{B_{R}} f^{\gamma_0+1}\eta^2\leq
2\gamma_0\alpha(n,\,p,\,\delta)R^2\int_{B_{R}}f^{\gamma_0+2}\eta^2+\gamma_0^3\left(\frac{\gamma_0^2}{\alpha(n,\,p,\,\delta)R^2}\right)^{\gamma_0+1}V.
\end{equation}

Now, we choose $\eta_0\in C_0^{\infty}\left(B_R\right)$ such that
$$
\begin{cases}
0\leq\eta_0\leq1,&\mbox{on}~B_{R},\\
\eta_0=1,&\mbox{on}~B_{3R/4},\\
\left|\nabla\eta_0\right|\leq\frac{8}{R}. &
\end{cases}
$$
And let $\eta=\eta_0^{\gamma_0+2}$, direct computation yields
$$
66R^2\int_{B_{R}}f^{\gamma_0+1}\left|\nabla\eta\right|^2\leq c\gamma_0^2\int_{B_{R}}f^{\gamma_0+1}\eta^{\frac{2(\gamma_0+1)}{\gamma_0+2}},
$$
where $c$ is a universal positive constant. Furthermore, by H\"{o}lder inequalityand Young's inequality, we have
$$
\begin{aligned}
c\gamma_0^2\int_{B_{R}}f^{\gamma_0+1}\eta^{\frac{2(\gamma_0+1)}{\gamma_0+2}}
\leq c\gamma_0^2\left( \int_{B_{R}}f^{\gamma_0+2}\eta^2\right)^{\frac{\gamma_0+1}{\gamma_0+2}}V^{\frac{1}{\gamma_0+2}}
\end{aligned}
$$
and
$$
\begin{aligned}
c\gamma_0^2\left( \int_{B_{R}}f^{\gamma_0+2}\eta^2\right)^{\frac{\gamma_0+1}{\gamma_0+2}}V^{\frac{1}{\gamma_0+2}}
\leq2\gamma_0\alpha(n,\,p,\,\delta)R^2\int_{B_{R}}f^{\gamma_0+2}\eta^2+ c^{\gamma_0+2}V \left(\frac{\gamma_0^2}{\alpha(n,\,p,\,\delta)R^2}\right)^{\gamma_0+1}.
\end{aligned}
$$
Therefore, we can derive
\begin{equation}\label{eqno4.13}
66R^2\int_{B_{R}}f^{\gamma_0+1}\left|\nabla\eta\right|^2\leq 2\gamma_0\alpha(n,\,p,\,\delta)R^2\int_{B_{R}}f^{\gamma_0+2}\eta^2+ c^{\gamma_0+2}V\left(\frac{\gamma_0^2}{\alpha(n,\,p,\,\delta)R^2}\right)^{\gamma_0+1}.
\end{equation}

Now, by substituting (\ref{eqno4.12}) and (\ref{eqno4.13}) into (\ref{eqno4.11}), we have
$$
\begin{aligned}
e^{-\gamma_0}V^{\frac{2}{n}}\left(\int_{B_{R}}f^{(\gamma_0+1)\lambda}\eta^{2\lambda}\right)^{\frac{1}{\lambda}}
\leq&\gamma_0^3\left(\frac{\gamma_0^2}{\alpha(n,\,p,\,\delta)R^2}\right)^{\gamma_0+1}V
+ c^{\gamma_0+2}V\left(\frac{\gamma_0^2}{\alpha(n,\,p,\,\delta)R^2}\right)^{\gamma_0+1}\\
=&V\left(\frac{\gamma_0^2}{\alpha(n,\,p,\,\delta)R^2}\right)^{\gamma_0+1}\left(\gamma_0^3+c^{\gamma_0+2}\right).
\end{aligned}
$$
This implies
$$
\begin{aligned}
\left(\int_{B_{R}}f^{(\gamma_0+1)\lambda}\eta^{2\lambda}\right)^{\frac{1}{\lambda}}
\leq&e^{\gamma_0}V^{1-\frac{2}{n}}\left(\frac{\gamma_0^2}{\alpha(n,\,p,\,\delta)R^2}\right)^{\gamma_0+1}\left(\gamma_0^3+c^{\gamma_0+2}\right).
\end{aligned}
$$
Thus, by a similar argument with the above section we have
$$
\begin{aligned}
\|f\|_{L^{\gamma_1}\left(B_{3R/4}\right)}
\leq&e^{\frac{\gamma_0}{\gamma_0+1}}V^{\frac{1}{\gamma_1}}\frac{\gamma_0^2}{\alpha(n,\,p,\,\delta)R^2}\left(\gamma_0^3
+c^{\gamma_0+2}\right)^{\frac{1}{\gamma_0+1}}\\
\leq& cV^{\frac{1}{\gamma_1}}\frac{\gamma_0^2}{\alpha(n,\,p,\,\delta)R^2},
\end{aligned}
$$
where $\gamma_1=(\gamma_0+1)\lambda$. The proof of Lemma \ref{lem4.2} is completed.
\end{proof}
{\bf Proof of Theorem \ref{thm2}} In the present situation, we can infer from (\ref{eqno4.9}) that
$$
\begin{aligned}
e^{-\gamma_0}V^{\frac{2}{n}}\left(\int_{B_{R}}f^{(\gamma+1)\lambda}\eta^{2\lambda}\right)^{\frac{1}{\lambda}}
\leq&66\int_{B_{R}} f^{\gamma+1}\left(\gamma_0^2\gamma \eta^2 +R^2\left|\nabla\eta\right|^2\right),
\end{aligned}
$$
if
$a>0,\,b>0,\, \max\{-a,\,-b\}<\mu<0,\,q\geq1$ and$$0<p<\frac{4(1-\delta)}{n-1},$$
where $\delta=\frac{\mu}{\max\{-a,\,-b\}}$. Hence, it follows that
\begin{equation}\label{eqno4.14}
\left(\int_{B_{R}}f^{(\gamma+1)\lambda}\eta^{2\lambda}\right)^{\frac{1}{\lambda}}
\leq 66e^{\gamma_0}V^{-\frac{2}{n}}\int_{B_{R}} f^{\gamma+1}\left(\gamma_0^2\gamma\eta^2  +R^2\left|\nabla\eta\right|^2\right).
\end{equation}
Now we take an increasing sequence $\{\gamma_k\}_{k=1}^{\infty}$ such that
$$
\gamma_1=(\gamma_0+1)\lambda\quad\mbox{and}\quad\gamma_{k+1}=\gamma_{k}\lambda,\quad k=1,\,2,\,...,
$$
and choose the same $\{r_k\}_{k=1}^{\infty}$  and $\{\eta_k\}_{k=1}^{\infty}\subset C_0^{\infty}(B_R)$ as those chosen in Section 3, i.e.,
$$
r_k=\frac{R}{2}+\frac{R}{4^k},\quad k=1,\,2,\,...,
$$
and
$$
\eta_k\in C_0^{\infty}(B_{r_k}),\quad \eta_k=1~\mbox{in}~B_{r_{k+1}}\quad\mbox{and}\quad\left|\nabla\eta_k\right|\leq\frac{4^{k+1}}{R}.
$$
By letting $\gamma+1=\gamma_k$ and $\eta=\eta_k$ in (\ref{eqno4.14}), we derive
$$
\begin{aligned}
\left(\int_{B_{R}}f^{\gamma_k\lambda}
\eta_k^{2\lambda}\right)^{\frac{1}{\lambda}}
\leq& c e^{\gamma_0}V^{-\frac{2}{n}}\gamma_0^3 16^k \int_{B_{r_{k}}} f^{\gamma_k}.
\end{aligned}
$$

Thus,
$$
\begin{aligned}
\left(\int_{B_{r_{k+1}}}f^{\gamma_{k+1}}\right)^{\frac{1}{\gamma_{k+1}}}
\leq& \left(c e^{\gamma_0}V^{-\frac{2}{n}}\gamma_0^3\right)^{\frac{1}{\gamma_k}} 16^{\frac{k}{\gamma_k}}\left( \int_{B_{r_{k}}} f^{\gamma_k}\right)^{\frac{1}{\gamma_k}},
\end{aligned}
$$
where $c$ is a universal positive constant. This means that
$$
\|f\|_{L^{\gamma_{k+1}}\left(B_{r_{k+1}}\right)}\leq
\left(ce^{\gamma_0}V^{-\frac{2}{n}}\gamma_0^3\right)^{\frac{1}{\gamma_{k}}}16^{\frac{k}{\gamma_k}} \|f\|_{L^{\gamma_{k}}\left(B_{r_{k}}\right)}.
$$
Hence, by iteration we have
\begin{equation}\label{eqno4.15}
\|f\|_{L^{\gamma_{k+1}}\left(B_{r_{k+1}}\right)}\leq
\left(ce^{\gamma_0}V^{-\frac{2}{n}}\gamma_0^3\right)^{\sum_{i=1}^{k}\frac{1}{\gamma_{i}}}16^{\sum_{i=1}^{k}\frac{i}{\gamma_i}} \|f\|_{L^{\gamma_{1}}\left(B_{3R/4}\right)}.
\end{equation}
As before we have
$$
\sum_{i=1}^{\infty}\frac{1}{\gamma_{i}}=\frac{n}{2\gamma_1}
\quad\quad~~\mbox{and}~~\quad\quad
\sum_{i=1}^{\infty}\frac{i}{\gamma_{i}}=\frac{n^2}{4\gamma_1}.
$$
In view of the above two facts, we let $k\rightarrow\infty$ in (\ref{eqno4.15}) to obtain the following
$$
\|f\|_{L^{\infty}\left(B_{R/2}\right)}\leq c(n)V^{-\frac{1}{\gamma_1}} \|f\|_{L^{\gamma_{1}}\left(B_{3R/4}\right)}.
$$
Hence, by Lemma \ref{lem4.2}, we conclude from the above inequality that
$$
\|f\|_{L^{\infty}\left(B_{R/2}\right)}\leq c(n) \frac{\gamma_0^2}{\alpha(n,\,p,\,\delta)R^2}.
$$
The definition of $\gamma_0$ and $\alpha(n,\,p,\,\delta)$ tell us that
$$
\|f\|_{L^{\infty}\left(B_{R/2}\right)}
\leq c\left(n,\,p,\,\delta\right)\frac{\left(1+\sqrt\kappa R\right)^2}{R^2}.
$$
Thus, we complete the proof of Theorem \ref{thm2}.\qed
\medskip

{\bf Proof of Corollary \ref{cor2}}
 Let $(M,\,g)$ be an n-dimensional noncompact complete Riemannian manifold with a nonnegative Ricci curvature. We assume $v$ is a smooth and positive solution of (\ref{eqno1.1}) with $a>0$, $b>0$, $\max\{-a,\,-b\}<\mu<0$. Then, for any $q\geq1$ and $p$ which satisfies that
 $$0< p<\frac{4(1-\delta)}{n-1},$$
where $\delta=\frac{\mu}{\max\{-a,\,-b\}}$, Theorem \ref{thm2} tells us that for any $B_{R}\subset M$, there holds
$$
\frac{\left|\nabla v\right|^2}{v^2}\leq\frac{c(n,\,p,\,\delta)}{R^2},\quad\mbox{on}~B_{R/2}.
$$
Now, letting $R\rightarrow\infty$ yields $\nabla v=0$. Therefore, $v$ is a positive constant on $M$.

On the other hand, when $a>0,\,b>0$, $\max\{-a,\,-b\}<\mu<0$,  $q\geq 1$ and $0< p<\frac{4\left(1-\delta\right)}{n-1}$, we have
$$
\Delta v +\mu v+av^{p+1}+bv^{-q+1}=v\left(\mu+av^{p}+bv^{-q}\right)\neq 0.
$$
This also means that $v$ could not be the solution to (\ref{eqno1.1}). Hence we know that (\ref{eqno1.1}) does not admit any positive solution. Thus we complete the proof of Corollary \ref{cor2}.
\medskip

As for Theorem \ref{thm3} and \ref{thm4} on the $n$-dimensional Einstein-scalar field Lichnerowicz equation $(n\geq3)$, we just need to let
$$p=\frac{4}{n-2}\quad\mbox{and}\quad q=\frac{4(n-1)}{n-2},$$
their proof follows immediately.

\section{Some applications and comments}
Now we recall some backgrounds on Lichnerowicz equation in relativity. A class of important problem is to construct solutions of the general relativistic vacuum constraint equations:
\begin{align}
R(g) = 2\Lambda + |K|_g^2 (\mbox{tr}_gK)^2,\label{5.1}\\
\mbox{div}_g K + \nabla \mbox{tr}_gK = 0,
\end{align}
so that the initial data set $(M, g, K)$ contains ends of cylindrical type, asymptotically flat type and asymptotically hyperbolic type. Here $K$ is a symmetric 2-tensor on $M$.

Fix a cosmological constant $\Lambda$ and a symmetric tensor field $\tilde{L}$ on a Riemannian manifold $(M, \tilde{g})?$, as well as a smooth bounded function $\tau$. This last function represents the trace of the extrinsic curvature tensor. We also do not need to assume that $\tilde{L}$ is transverse traceless, though it is in the application to the constraint equations. To simplify
notation, set
$$\tilde{\sigma}^2 := \frac{n-2}{4(n-1)}|\tilde{L}|^2_{\tilde{g}}, \quad\quad \beta := \left[\frac{n-2}{4n}\tau^2 -\frac{n-2}{2(n-1)} \Lambda\right],$$
which is convenient since only these quantities, rather than $\tilde{L}$,  $\tau$ or $\Lambda$, appear in the constraint equations. The symbol $\tilde{\sigma}^2$ is meant to remind the reader that this function is nonnegative, but is slightly misleading since $\tilde{\sigma}^2$ may not be the square of a smooth function. All of these functions are as regular as the metric $g$ and the extrinsic curvature $K$.

The following equation ia usually called (Einstein-scalar field) Lichnerowicz equation
$$\Delta_{\tilde{g}}\tilde{\phi}- \frac{n-2}{4(n-1)}\tilde{R}\tilde{\phi}= \beta\tilde{\phi}^{(n+2)/(n-2)}- \tilde{\sigma}^2\tilde{\phi}^{-(3n-2)/(n-2)},$$
which corresponds to \eqref{5.1}. For the sake of convenience, one likes to write it more simply as
\begin{align}\label{5.3}
L_{\tilde{g}}\tilde{\phi}=\beta\tilde{\phi}^\alpha -\tilde{\sigma}^2\tilde{\phi}^{-\gamma}.
\end{align}
Here, $L_{\tilde{g}}$ denotes the conformal Laplacian,
$$L_{\tilde{g}} = \Delta_{\tilde{g}}- \frac{n - 2}{4(n-1)}\tilde{R}$$
and one also always set
$$c(n) = \frac{n - 2}{4(n-1)}, \quad \alpha = \frac{n + 2}{n-2}\quad\text{and}\quad \gamma = \frac{3n-2}{n-2}.$$
An important property of \eqref{5.3} is the following conformal transformation property. Suppose that $\hat{g} = u^{\frac{4}{(n-2)}}\tilde{g}$. It is well known that for any function $\hat{\phi}$, it
holds that
$$L_{\tilde{g}} (u\hat{\phi}) = u^\alpha L_{\hat{g}}(\hat{\phi}),$$
where $L_{\hat{g}}$ is the conformal Laplacian associated with $\hat{g}$. Thus, if we substitute $\tilde{\phi}= u\hat{\phi}$ into \eqref{5.3}, this last equation becomes
$$u^\alpha L_{\hat{g}}\hat{\phi} = \beta u^\alpha \hat{\phi}^\alpha - \tilde{\sigma}^2u^{-\gamma}\hat{\phi}^{-\gamma}.$$
Dividing by $u^\alpha$ and defining
$$\hat{\sigma}^2 = u^{-\gamma-\alpha}\tilde{\sigma}^2,$$
then we have simply that
$$L_{\hat{g}}\hat{\phi} = \beta\hat{\phi}^\alpha - \hat{\sigma}^2\hat{\phi}^{-\gamma}.$$
Note in particular that while ?$\tilde{\sigma}^2$ transforms by a power of $u$, the coefficient
function $\beta$ is the same in the transformed equation.

We know that there are several classes of general relativistic initial data sets, which have been extensively studied, including: compact manifolds, manifolds with asymptotically flat ends, and manifolds with asymptotically hyperbolic ends.

There are also another types of interesting asymptotic geometry, which appear naturally in general relativistic studies, namely: manifolds with ends of cylindrical type, asymptotically periodic type, with cylindrically bounded type. For instance, P.T. Chru\'sciel and R. Mazzeo \cite{CM} have ever studied the existence of solutions $\tilde{\phi}$ to \eqref{5.3} with controlled asymptotic behavior in the following cases:

A. $(M,\tilde{g})$ is a complete manifold with a finite number of ends of cylindrical type;

B. $(M,\tilde{g})$ is a complete manifold with a finite number of asymptotically conic ends and a finite number of ends of cylindrical type, with $\tilde{R}\geq 0$;

C. $(M,\tilde{g})$ is a complete manifold with a finite number of ends of cylindrical type and a finite number of asymptotically hyperbolic ends. In this case, we further assume that $\tilde{R}<0$ sufficiently far out on all ends and hence
require that $\beta > 0$ if $\tau = const$.

For more details we refer to Theorem 4.4 and Theorem 4.5 in \cite{CM}.
\medskip

It is not difficult to see that  Theorem \ref{thm3} implies the following

\begin{thm}
Suppose that $\beta > 0$ and $\tilde{\sigma}^2$ are two positive constants, and $B_{2R}\in M$ be a geodesic ball, where $(M,\,\tilde{g})$ is a complete Riemannian manifold with $Ric_{\tilde{g}} \geq -(n-1)\kappa$. Assume that the scalar curvature $\tilde{R}|_{B_{2R}}\leq 0$ is a constant. If $v$ is a smooth positive solution of the Einstein-scalar field Lichnerowicz equation \eqref{5.3} on the geodesic ball $B_{2R}\subset M$, then, on $B_{R}$ there holds
$$
\frac{\left|\nabla v\right|^2}{v^2}\leq  c(n, \tilde{R}, \beta, \tilde{\sigma}^2)\frac{\left(1+\sqrt\kappa R\right)^2}{R^2},
$$
where $c(n, \tilde{R}, \beta, \tilde{\sigma}^2)$ is a constant.

Moreover, if $(M,\,g)$ is a noncompact complete Ricci flat manifold, $\beta > 0$ and $\tilde{\sigma}^2$ are constants, then any positive solution to the Einstein-scalar field Lichnerowicz equation is a constant.
\end{thm}

How to get a similar gradient estimate for the positive solution to Lichnerowicz equation \eqref{5.3} with variable coefficients is a hard problem. In a forthcoming paper we will discuss this problem.

\medskip

{\bf Acknowledgments:} Y. Wang is supported partially by National Natural Science Foundation of China (Grant No.11971400) and National Key Research and Development Projects of China (Grant No.2020YFA0712500).


\bigskip

\begin{thebibliography}{MM}
\bibitem{CGS} L. Caffarelli, B. Gidas and J. Spruck, \emph{Asymptotic symmetry and local behavior of semilinear elliptic equations with critical Sobolev growth}, Comm. Pure Appl. Math. {\bf 42}(1989), no.3, 271-297.


\bibitem{CY} S. Y. Cheng and S. T. Yau, \emph{Differential equations on Riemannian manifolds and their geometric applications}, Comm. Pure Appl. Math. {\bf 28}(1975), no. 3, 333-354.


\bibitem{S. M. Choi, J.-H. Kwon & Y. J. Suh1998} S. M. Choi, J.-H. Kwon and Y. J. Suh, \emph{A Liouville-type theorem for complete Riemannian manifolds}, Bull. Korean Math. Soc. {\bf 35}(1998), no.2, 301-309.


\bibitem{C} Y. Choquet-Bruhat, \emph{General Relativity and the Einstein Equations}, Oxford Mathematical Monographs, Oxford University Press, Oxford (2009)

\bibitem{CGi} P.T. Chru\'sciel and R. Gicquaud, \emph{Bifurcating solutions of the Lichnerowicz equation}, Ann. Henri Poincar$\acute{e}$, {\bf 18}(2017), no. 2, 643-679.

\bibitem{CM} P.T. Chru\'sciel and R. Mazzeo, \emph{Initial data sets with ends of cylindrical type: I. The Lichnerowicz equation}, Ann.H. Poincar$\acute{e}$, {\bf 16}(2014), 1231-1266.

\bibitem{CCKW} A. Constantin, D.G. Crowdy, V.S. Krishnamurthy and M.H. Wheeler, \emph{Stuart-type polar vortices on a rotating sphere,} Discrete Contin. Dyn. Syst. {\bf 41} (2021), 201--215.

\bibitem{CG} A. Constantin and P. Germain, \emph{Stratospheric planetary flows from the perspective of the Euler equation on a rotating sphere,} Arch. Ration. Mech. Anal. {\bf 245} (2022), 587--644.

\bibitem{DG} Y. Du and Z.-M. Guo, \emph{Positive solutions of an elliptic equation with negative exponent: stability and critical power}, J. Differential Equations {\bf 246} (2009), no.6, 2387-2414.

\bibitem{DKN} N.T. Dung, N.N. Khanh and Q.A. Ng\^o, \emph{Gradient estimates for some f-heat equations driven by Lichnerowicz's equation on complete smooth metric measure spaces}, Manuscripta Math. {\bf155}(2018), 471-501.

\bibitem{GKS} M. Ghergu, S. Kim and H. Shahgholian, \emph{Exact behaviour around isolated singularity for semilinear elliptic equations with a log-type nonlinearity}, arXiv:1804.04287. to appear in Adv. Nonlinear Anal.

\bibitem{He-Wang-Wei} J. He, Y.-D. Wang and G.-D. Wei; \emph{Gradient estimates for $\Delta_pu + av^q = 0$ on a complete Riemannian manifold and Liouville type theorems}, preprint, arXiv: 2304.08238.


\bibitem{HPP} E. Hebey, F. Pacard and D. Pollack, \emph{A variational analysis of Einstein-Scalar field Lichnerowicz equations on compact Riemannian manifolds}, Commun. Math. Phys. {\bf278}(2008), 117-132.

\bibitem{P. Huang & Y. Wang2022}P. Huang, Y. Wang, \emph{Gradient estimates and Liouville theorems for Lichnerowicz equations}, Pacific J. Math. {\bf 317}(2022), no.2, 363-386.

\bibitem{I} J. Isenberg, \emph{Constant mean curvature solutions of the Einstein constraint equations on closed manifolds}. Class. Quantum
Grav. {\bf 12}(1995), 2249-2274.

\bibitem{IMP} J. Isenberg, D. Maxwell and D. Pollack, \emph{A gluing constructions for non-vacuum solutions of the Einstein constraint
equations}. Adv. Theor. Math. Phys. {\bf 9}(2005), 129-172.

\bibitem{LY} P. Li and S.T. Yau, \emph{On the parabolic kernel of the Schr$\ddot{o}$dinger operator}, Acta. Math. {\bf 156(3)}(1986), 153-201.

\bibitem{Lic}  A. Lichnerowicz, \emph{L'int\'egration des \'equations de la gravitation relativiste et
le probl$\grave{e}$me des $n$ corps}, J. Math. pures appliqu\'ees {\bf23}(1944), 37-63.

\bibitem{MHL} B.-Q. Ma, G.-Y. Huang and Y. Luo, \emph{Gradient estimates for a nonlinear elliptic equation on complete Riemannian manifolds}, Proc. Amer. Math. Soc. {\bf 146}(2018),  4993-5002.

\bibitem{M} L. Ma, \emph{Gradient estimates for a simple elliptic equation on complete non-compact Riemannian manifolds}, J. Funct. Anal. {\bf 241}(2006), 374-382.


\bibitem{MW} L. Ma and J.-C. Wei, \emph{Properties of positive solutions to an elliptic equation with negative exponent}, J. Funct.
Anal. {\bf 254}(2008), 1058-1087.


\bibitem{MW1} L. Ma and J.-C. Wei, \emph{Stability and multiple solutions to Einstein-scalar field Lichnerowicz equation on manifolds}, J. Math. Pures Appl. {\bf 99(9)}(2013), no.2, 174-186.

\bibitem{N} Q.A. Ng\^{o}, \emph{Einstein constraint equations on Riemannian manifolds. In: Geometric Analysis Around Scalar Curvatures}, vol.31, pp.119-210. Lecture Notes Series, Institute for Mathematical Sciences, National University of Singapore, World Scientific(2016)

\bibitem{PWW} B. Peng, Y.-D. Wang and G.-D. Wei, \emph{Gradient estimates and Liouville theorems for $\Delta u+au^{p+1}=0$}, Mathematical Theory and Application {\bf 43}(2023), 32-43.


\bibitem{PWW1} B. Peng, Y.-D. Wang and G.-D. Wei, \emph{Yau type gradient estimates for $\Delta u+au(\log u)^{p+1}=0$ on Riemannian manifolds}, J. Math. Anal. Appl. {\bf 498}(2021), no. 1, Paper No. 124963, 24 pp.

\bibitem{PQS} P. Pol\'a$\check{c}$ik, P. Quittner and P. Souplet, \emph{Singularity and decay estimates in superlinear problems via Liouville-type theorems, I: elliptic equations and systems}, Duke MATH. J., {\bf 139}(2007), No. 3, 555-574.

\bibitem{R} L. Rudnicki, \emph{Geophysics and Stuart vortices on a sphere meet differential geometry}, Commun. Pure Appl. Anal. {\bf 21}(2022), 2479-2493.


\bibitem{L. Saloff-Coste1992}L. Saloff-Coste, \emph{Uniformly elliptic operators on Riemannian manifolds}, J. Differential Geom. {\bf36}(1992), no.2, 417 - 450.

\bibitem{S1} R. Schoen, \emph{Conformal deformation of a Riemannian metric to constant scalar curvature}, J. Diff. Geometry {\bf20}(1984), 479-495.

\bibitem{S2} R. Schoen, \emph{The existence fo weak solutions with prescribed singular behavior for a conformally invariant scalar equation}, Comm. Pure Appl. Math. {\bf41}(1988), 317-392.

\bibitem{SY} R. Schoen and S. T. Yau, \emph{Lectures on Differential Geometry}, International Press, Cambridge, MA, (1994).

\bibitem{Sou} P. Souplet, \emph{The proof of the Lane-Emden conjecture in four space dimensions}, Advances in Mathematics, {\bf 221}(2009), 1409-1427.

\bibitem{S. T. Yau1975}S.T. Yau, \emph{Harmonic functions on complete Riemannian manifolds}, Comm. Pure Appl. Math. {\bf28}(1975), 201-228.

\bibitem{SZ} X.-F. Song and L. Zhao, \emph{Gradient estimates for the elliptic and parabolic Lichnerowicz equations on compact manifolds}, Z. Angew. Math. Phys. {\bf61}(2010), no.4, 655-662.


\bibitem{W} F.-Y. Wang, \emph{Harnack inequalities for log-Sobolev functions and estimates of log-Sobolev constants}, Ann.
Probab. {\bf27(2)}(1999), 653-663.

\bibitem{WZ}
X.-D. Wang and L. Zhang, \emph{Local gradient estimate for {$p$}-harmonic functions on Riemannian manifolds}, Comm. Anal. Geom. {\bf 19}(2011), no.4, 759-771.

\bibitem{WW}
Y.-D. Wang and G.-D. Wei,\emph{On the nonexistence of positive solutions to $\Delta u + au^{p+1}=0$ on Riemannian manifolds}. J. Differential Equations {\bf362}(2023), 74-87.

\bibitem{W1} J.-Y. Wu, \emph{Gradient estimates for a nonlinear parabolic equation and Liouville theorems}, Manuscripta Math. {\bf159}(2019), no.3-4, 511-547.

\bibitem{Y-Z} F. Yang and L.-D. Zhang, \emph{Gradient estimates for a nonlinear parabolic equation on smooth metric measure spaces}, Nonlinear Analysis: Theory, Methods \& Applications {\bf187}(2019), 49-70.


\bibitem{York} J. York, \emph{Conformally invariant orthogonal decomposition of symmetric
tensors on Riemannian manifolds and the initial-value problem of general relativity}, J. Math. Phys. {\bf14}(1973), No.4, 456-464.


\bibitem{ZY} L. Zhao and D.-Y. Yang, \emph{Gradient estimates for the p-Laplacian Lichnerowicz Equation on smooth metric measure spaces}, Proc. of the American Mathe. Society {\bf146}(2018), 5451-5461.

\bibitem{Z} L. Zhao, \emph{Liouville theorem for Lichnerowicz equation on complete noncompact manifolds}, Funkcial. Ekvac. {\bf 57}(2014), no.1, 163-172.

\bibitem{Y.-Y. Yang2010}Y.-Y. Yang, Gradient estimates for the equation $\Delta u+cu^{-\alpha}=0$  on Riemannian manifolds, Acta. Math. Sin. {\bf 26}(2010), no.6, 1177-1182.

\end{thebibliography}
\end{document}